\documentclass[11pt, leqno]{amsart}
\usepackage{amsfonts,amssymb, amscd}
\usepackage{bm}
\usepackage{amsmath}
\usepackage{lmodern}
\usepackage{mathrsfs}
\usepackage{amsthm}
\usepackage{qsymbols}
\usepackage{mathtools}
\mathtoolsset{showonlyrefs}
\usepackage{dsfont}
\usepackage{graphicx}
\usepackage{latexsym}
\usepackage{chngcntr}
\usepackage{paralist}
\usepackage[parfill]{parskip}
\usepackage{bm}
\usepackage{tikz-cd}
\usepackage{tikz}
\usetikzlibrary{arrows,intersections}
\usetikzlibrary{decorations.fractals}
\usepackage{csquotes}
\usepackage{esint}
\usepackage{nicefrac}

\usetikzlibrary{decorations.pathreplacing}

\let\oldproofname=\proofname
\renewcommand{\proofname}{\rm\bf{\oldproofname}}

\usepackage[shortlabels]{enumitem}
\newtheorem{theorem}{Theorem}[section]
\newtheorem{lemma}[theorem]{Lemma}
\newtheorem{proposition}[theorem]{Proposition}
\newtheorem{corollary}[theorem]{Corollary}
\newtheorem*{thm}{Theorem}
\theoremstyle{definition}
\newtheorem*{assumptions}{Assumptions}
\newtheorem{assumption}[theorem]{Assumption}
\newtheorem{definition}[theorem]{Definition}
\newtheorem{remark}[theorem]{Remark}
\newtheorem{example}[theorem]{Example}
\usepackage[plainpages=false,pdfpagelabels,
backref=page,
citecolor=red]{hyperref}
\usepackage{xcolor}
\hypersetup{
	colorlinks,
	linkcolor={cyan!90!black},
	citecolor={magenta},
	urlcolor={green!40!black}}

\usepackage[disable]{todonotes}
\DeclareMathOperator{\loc}{loc}

\DeclareMathOperator{\e}{e}

\DeclareMathOperator{\Div}{div}

\DeclareMathOperator{\supp}{supp}

\DeclareMathOperator{\cc}{c}

\newcommand{\C}{\mathbb{C}}
\newcommand{\R}{\mathbb{R}}
\newcommand{\N}{\mathbb{N}}

\newcommand{\Z}{\mathbb{Z}}
\renewcommand{\H}{\mathrm{H}}
\newcommand{\W}{\mathrm{W}}

\DeclareRobustCommand{\Cdot}{\dot{\rC}\protect{\vphantom{C}}}

\renewcommand{\L}{\mathrm{L}}
\newcommand{\rC}{\mathrm{C}}

\newcommand{\VMO}{\mathrm{VMO}}
\newcommand{\BMO}{\mathrm{BMO}}
\newcommand{\D}{\mathsf{D}}
\newcommand{\sR}{\mathsf{R}}
\newcommand{\sN}{\mathsf{N}}

\newcommand{\ball}{\mathrm{B}}

\renewcommand{\d}{\mathrm{d}}
\newcommand{\diam}{\mathrm{diam}}

\newcommand{\E}{\mathcal{E}}

\newcommand{\les}{\lesssim}

\DeclareMathOperator{\re}{Re}
\renewcommand{\t}{\boldsymbol{t}}

\newcommand{\1}{\boldsymbol{1}}

\newcommand{\sub}{\subseteq}

\newcommand{\Miyachi}{\mathrm{Mi}}
\newcommand{\CW}{\mathrm{CW}}
\newcommand{\restrict}{r}

\newcommand{\zero}{z}
\newcommand{\glob}{\mathrm{glob}}
\newcommand{\dO}{{\mathrm{d}(O)}}
\newcommand{\OD}{O_D}
\newcommand{\ON}{O_N}
\newcommand{\ISet}{I} %
\renewcommand{\a}{\alpha}
\renewcommand{\j}{\varphi} %
\newcommand*\dif{\mathop{}\!\mathrm{d}}
\newcommand{\smooth}[1][]{\rC_{\cc}^{\infty}(#1)}

\def\XXint#1#2#3{{\setbox0=\hbox{$#1{#2#3}{%
				\int}$ }
		\vcenter{\hbox{$#2#3$ }}\kern-.6\wd0}}
\title[Hardy spaces adapted to elliptic operators on open sets]{Hardy spaces adapted to elliptic operators \\ on open sets}
\author{Sebastian Bechtel}
\author{Tim B\"ohnlein}

\address{Delft Institute of Applied Mathematics, Delft University of Technology, P.O. Box 5031, 2600 GA Delft, The Netherlands}

\email{S.Bechtel@tudelft.nl}

\address{Fachbereich Mathematik, Technische Universit\"at Darmstadt, Schlossgartenstr. 7, 64289 Darmstadt, Germany}

\email{boehnlein@mathematik.tu-darmstadt.de}

\subjclass[2020]{Primary: 35J25, 47F10. Secondary: 35B65, 46E35.}
\date{\today}
\dedicatory{}
\keywords{Hardy spaces, atomic spaces, elliptic operators, square functions, geometric analysis}
\begin{document}
	\begin{abstract}
		Let $L= - \Div (A \nabla \cdot)$ be an elliptic operator defined on an open subset of $\R^d$, complemented with mixed boundary conditions. Under suitable assumptions on the operator and the geometry, we derive an atomic characterization (depending only on the boundary conditions) for the Hardy space $\H^1_L$ defined using an adapted square function for $L$. This generalizes known results of Auscher and Russ~\cite{Auscher-Russ} in the case of pure Dirichlet/Neumann boundary conditions on Lipschitz domains. In particular, we develop a connection between the harmonic analysis of $L$ and its underlying geometry.
	\end{abstract}
	\maketitle

	\makeatletter
	\providecommand\@dotsep{5}
	\makeatother
	\listoftodos\relax

	\section{Introduction} \label{Section: Introduction}

	Let us start with the classical Hardy space $\H^1(\R^d)$. There are many equivalent descriptions for it, including atomic, maximal and square function characterizations. All these perspectives are valuable and prove their merit in various applications. For instance, both the maximal and square function characterizations are essential tools in the treatment of harmonic boundary value problems on the half-space: the maximal characterization allows a better understanding in which sense boundary values can actually be attained, whereas the square function characterization gives rise to solutions of boundary value problems using techniques from evolution equations. The atomic characterization makes cancellation properties of the Hardy space visible and accessible in the treatment of general singular integrals.

	The space $\H^1(\R^d)$ can be seen as a space adapted to the negative Laplacian $-\Delta$ on the Euclidean space $\R^d$. Consequently, if one is interested in studying boundary value problems with respect to other elliptic operators than the Laplacian~\cite{Block}, or maybe even boundary value problems in cylindrical domains~\cite{AE-mixed}, it seems logical to study Hardy spaces that are adapted to a different elliptic operator than the Laplacian.

	Let $d \geq 2$, $O \sub \R^d$ be open, and $A$ a complex, bounded, elliptic coefficient matrix on $O$. Consider the elliptic operator in divergence form $L = - \Div (A \nabla \cdot)$ subject to mixed boundary conditions.
	For the moment, the reader can keep a bounded domain subject to pure Dirichlet or Neumann boundary conditions in mind. Nevertheless, we mention already here that the interplay between different boundary conditions and components of $O$ leads to interesting effects that we are going to discuss in a minute. For the precise geometric and operator theoretic setup we refer to Section~\ref{Section: Operator theoretic and geometric setup}.
	We consider the operator adapted Hardy space $\H^1_L \coloneqq \H^1_L(O)$ that consists of all $f\in \L^1(O)$ such that the square function
	\begin{align}
		\label{eq:intro_sqf}
		(S_L f)(x) = \Bigl( \int_0^\infty \fint_{\ball(x,t) \cap O} |t \partial_t \e^{-t \sqrt{L}} f(y)|^2 \,\d y \,\frac{\d t}{t} \Bigr)^\frac{1}{2}, \qquad x\in O,
	\end{align}
	is again in $\L^1(O)$. We are going to omit the underlying set $O$ in the notation of function space from now on.

	In the application to boundary value problems, adapted Hardy spaces describe the space of admissible boundary data. Therefore, it is desirable to have a description of $\H^1_L$ which is independent of the concrete operator $L$. Atomic spaces can provide such a description.
	Examples of atomic Hardy spaces on open sets include the space $\H^1_\CW$ studied by Coifman--Weiss~\cite{Coifman-Weiss} and $\H^1_\Miyachi$ of Miyachi~\cite{Miyachi-Atoms}. It turns out~\cite{Auscher-Russ,JFA-Hardy} that the space $\H^1_\CW$ corresponds to pure Neumann and $\H^1_\Miyachi$ to pure Dirichlet boundary conditions, in other words: under suitable assumptions, $\H^1_\CW$ coincides with $\H^1_L$ if $L$ is subject to pure Neumann boundary conditions and $\H^1_\Miyachi$ coincides with $\H^1_L$ if $L$ is subject to pure Dirichlet boundary condition. We are going to review Hardy spaces on open sets in more depth in Section~\ref{Section: Review of Hardy spaces on open sets}.
	To the best of our knowledge, if the matrix $A$ is not self-adjoint then only the results from~\cite{Auscher-Russ} relate square function based spaces with ($L$-independent) atomic spaces.
	Their main geometric assumption is that $O$ is connected with Lipschitz boundary.
	The purpose of this paper is a generalization of~\cite{Auscher-Russ} in the following two directions:
	\begin{enumerate}
		\item Can we work with a rough and unconnected open set $O$ far below the Lipschitz class?
		\item Can we replace pure Dirichlet and Neumann boundary conditions by mixed boundary conditions?
	\end{enumerate}
	Let us start with an example that highlights that some geometric conditions are necessary.
	\begin{example}
		\label{eq:punctured_plane}
		Consider the punctured plane $O = \R^2 \setminus \{ 0 \}$ with boundary $\partial O = \{ 0 \}$. Then $\W^{1,2}_0(O) = \W^{1,2}(\R^2)$. Hence, the negative Dirichlet Laplacian $-\Delta_0$ on $O$ coincides with the usual negative Laplacian $-\Delta$ on $\R^2$. Consequently, $\H^1_{-\Delta_0} = \H^1(\R^2)$ by~\cite{Auscher-Russ}. But $\H^1_\Miyachi$ contains functions that are not mean value free, therefore $\H^1_{-\Delta_0} \neq \H^1_\Miyachi$. This show that a fatness condition for $\partial O$ cannot be avoided.
	\end{example}
	Pushing the example further, if $D \subseteq \R^d$ is of Hausdorff dimension less than $d-2$, then an application of Netrusov's theorem~\cite[Ch.~10]{Hedberg} gives a counterexample for $O = \R^d \setminus D$. On the other hand, if the boundary has dimension larger than $d-2$, our theory gives an answer to the affirmative.

	As a first gentle step into mixed boundary conditions, consider an open, bounded set $O = O_1 \cup O_2$ that consists of two separated connected components. Consider an operator $L$ on $O$ subject to pure Dirichlet boundary conditions on one component, and pure Neumann boundary conditions on the other component. A natural guess for the corresponding atomic Hardy space is
	\begin{align}
		\label{eq:intro_atom_decoupling}
		\H^1_D(O) = \H^1_\CW(O_1) \oplus \H^1_\Miyachi(O_2),
	\end{align}
	where $O_1$ is the Neumann component and $O_2$ is the Dirichlet component. Taking a look at the square function~\eqref{eq:intro_sqf}, such a decoupling is not to be expected as the square function always integrates over both components for $t$ large enough.
	Nevertheless,~\eqref{eq:intro_atom_decoupling} was the correct guess and with its help we eventually even find
	\begin{align}
		\H^1_D(O) = \H^1_L = \H^1_{L_1} \oplus \H^1_{L_2},
	\end{align}
	where $L_1$ and $L_2$ are restrictions of the elliptic operator $L$ to the components $O_1$ and $O_2$.
	Going one step further, we even want to study the situation where $D \subseteq \partial O$ is a non-trivial subset of the boundary of a, say connected, open set $O$. We have indicated in Example~\ref{eq:punctured_plane} that Dirichlet boundary conditions can lead to a lack of cancellation for atoms near the boundary in the sense that they are not mean value free anymore, in contrast to the case of classical atoms in the sense of Coifman--Weiss. The right definition for an atomic space $\H^1_D$ therefore has to \enquote{interpolate} the definitions for $\H^1_\CW$ and $\H^1_\Miyachi$ in a non-trivial way. We comment in Remark~\ref{rem:miyachi} on the upcoming challenges in this quest. Eventually, we come up with a suitable notion of $\H^1_D$-atoms in Definition~\ref{def:H1D_atoms}.
	The abstract main result of this article is the following. The exact assumptions on geometry and elliptic operator will be introduced in Section~\ref{Section: Operator theoretic and geometric setup} and we  give a precise formulation of the result in Theorem~\ref{Main result and applications: Theorem: H_D^1=H_L^1}.
	\begin{thm}
		Let $O \subseteq \R^d$ open and $D \subseteq \partial O$ closed.
		Assume that $O$ is interior thick and locally uniform near $N \coloneqq \partial O \setminus D$, that $D$ is porous and that there is some fatness condition away from $N$. Suppose that $L$ satisfies Hölder-type Gaussian estimates. Then
		\begin{align}
			\H^1_L = \H^1_D.
		\end{align}
	\end{thm}
	We give some sufficient conditions for our geometric framework.
	\begin{example}
		\label{ex:intro_geom}
		Interior thickness ensures that $O$ is at least a space of homogeneous type. Being a locally uniform domain near $N$ is a generalization of Jones' notion of an $(\varepsilon,\delta)$-domains~\cite{Jones}. All open sets whose boundary is represented by a graph are admissible, but also fractals like the von Koch snowflake, see Figure~\ref{fig:snowflake}. The condition for $D$ means that the Dirichlet part is not of the same dimension as $O$ and is essentially always fulfilled. It is in particular the case when $D$ is Ahlfors--David regular, in which case also the fatness condition away from the Neumann part follows. Also exterior thick sets (see Proposition~\ref{prop:restrict}) satisfy the fatness condition.
	\end{example}

	Our definition also unifies the cases of pure Dirichlet and Neumann boundary conditions. See Proposition~\ref{prop:atomic_pure_cases} for the precise assumptions.
	\begin{thm}
		Under minimal geometric conditions one has
		\begin{align}
			\H^1_\varnothing = \H^1_\CW \qquad \& \qquad \H^1_{\partial O} = \H^1_\Miyachi.
		\end{align}
	\end{thm}
	An essential tool to prove our main result is a full duality theory for $\H^1_D$. We are going to introduce $D$-adapted spaces $\VMO_D$ and $\BMO_D$. A key feature of our approach is the fact that the $\BMO$-seminorm is defined by bounded means and not bounded mean oscillation near the Dirichlet boundary part. The following theorem recaptures our duality theory for $\H^1_D$. Precise versions are given in Theorems~\ref{H1-BMO-duality: Theorem: Main result} and~\ref{Predual of H1: Theorem: VMO* = H1}.
	\begin{thm}
		Under minimal geometric assumptions one has
		\begin{align}
			(\VMO_D)^* = \H^1_D \qquad \& \qquad (\H^1_D)^* = \BMO_D.
		\end{align}
	\end{thm}
	We have seen in Example~\ref{ex:intro_geom} that our geometric framework comprises many relevant situations. So far, we have not discussed the Hölder-type Gaussian estimates (see Definition~\ref{Property G(mu): Definition: G(mu)}) that are imposed on $L$ in our main result.
	In Lipschitz domains they were pioneered in~\cite{Auscher_Tchamitchian_Domains}.
	Recently, they were established in an article by the second-named author together with Ciani and Egert~\cite{BCE-Gauss-for-MBC} for mixed boundary conditions.
	As a combination of our main result, the Gaussian estimates from~\cite{BCE-Gauss-for-MBC} and the recent article~\cite{JFA-Hardy} in a self-adjoint setting, we give a full picture (including maximal characterizations) for $-\Delta$ subject to mixed boundary conditions on a bounded open set. The following theorem contains the special case where $O$ is connected and $D$ is non-trivial. The precise formulation in the general case can be found in Theorem~\ref{thm:laplace}.
	\begin{thm}
		Let $O \subseteq \R^d$ be a bounded domain that is interior thick and let $D\subseteq \partial O$ be non-trivial, closed and porous. Suppose that $O$ is locally uniform near $N \coloneqq \partial O \setminus D$ and fat away from $N$. Then
		\begin{align}
			\{ f\in \L^1 \colon S_{-\Delta_D}(f) \in \L^1 \} = \H^1_D = \H^1_{-\Delta_D,\mathrm{max}} = \H^1_{-\Delta_D,\mathrm{rad}},
		\end{align}
		where $-\Delta_D$ is the negative Laplacian subject to mixed boundary conditions. The spaces $\H^1_{-\Delta_D,\mathrm{max}}$ and $\H^1_{-\Delta_D,\mathrm{rad}}$ are non-tangential and radial maximal spaces.
	\end{thm}
	In the preceding result we can always replace the Laplacian by an elliptic operator with a real and self-adjoint coefficient matrix. In $d=2$ the first identity stays true if all restrictions but ellipticity on the coefficients are dropped~\cite[Thm.~10.1]{BCE-Gauss-for-MBC}.

	Let us emphasize that all assumptions of the last theorem as well as the space $\H^1_D$ appearing in its conclusion are purely geometrical.

\subsection*{Notation}

The space dimension $d \geq 2$ is fixed and $O \subseteq \R^d$ is open with diameter $\dO \coloneqq \diam(O)$. Fix also some closed set $D\subseteq \partial O$ called the \emph{Dirichlet part}. Abbreviate function spaces over $O$ by $\mathrm{X} \coloneqq \mathrm{X} (O)$, for instance write $\L^2$ instead of $\L^2(O)$. Let $\ball = \ball(x,r) \sub \R^d$ denote the Euclidean ball with center $x \in \R^d$ and radius $r = r(\ball) > 0$. For $c > 0$ write $c\ball(x,r) = \ball(x, cr)$ and for $j \geq 1$ put $C_j(\ball) \coloneqq 2^{j+1} \ball \setminus 2^j \ball$. If $E\subseteq \R^d$ set $E(x,r) \coloneqq \ball(x,r) \cap E$. Let $\d(E,F)$ denote the distance between sets $E, F\subseteq \R^d$. Write $\boldsymbol{t}$ and $\boldsymbol{z}$ for the identity functions $t \mapsto t$ and $z\mapsto z$ defined either on a real interval or a complex sector, respectively. If not mentioned otherwise, all space equalities are understood in the set theoretic sense with equivalence of norms.

\section*{Acknowledgment}
The first-named author was partially supported by the ANR project \enquote{RAGE: ANR-18-CE-0012-01} and a \enquote{Feodor--Lynen} grant of the Humboldt foundation. The second-named author thanks the Technical University of Delft for hospitality during a research stay in March 2023. The first-named author is grateful to Pascal Auscher who proposed the subject to him.

	\section{Geometric framework and the elliptic operator}  \label{Section: Operator theoretic and geometric setup}

	\subsection{The geometric setup}  \label{Subsection: The geometric setup}

	We start out with a list of geometric concepts.

	\begin{assumptions} \label{Assumption: H_D^1=H_L^1}
		Throughout this paper we make use of the following assumptions.
		\begin{enumerate}

			\item[(UITC)] $O$ is \textbf{uniformly interior thick}:
			\begin{equation}  \label{item:ITC}
				\exists \, c > 0 \; \forall \, x \in O, r < \dO \colon \quad |\ball(x,r) \cap O| \geq c r^d.
			\end{equation}

			\item[$(D)$] $D$ is \textbf{uniformly porous}:
			\begin{equation*}
				\exists \, \kappa \in (0,1] \; \forall \, x \in D, r < \dO \, \exists \, y \in \R^d \colon \quad  \ball(y, \kappa r) \sub \ball(x,r) \setminus \partial O.
			\end{equation*}
			\item[(LU)] $O$ is \textbf{locally uniform near $\boldsymbol{N}$}. There are $\varepsilon \in (0, 1]$ and $\delta \in (0, \infty]$ such that the following properties hold.
			\begin{enumerate}
				\item[(i)] Set $N_\delta = \{ x \in \R^d \colon \d(x, N) < \delta \}$. All points $x,y \in N_{\delta} \cap O$ with $|x-y| < \delta$ can be joined in $O$ by an $\varepsilon$-cigar with respect to $\partial N_{\delta} \cap O$: There is a rectifiable curve $\gamma \sub O$ of length $\ell(\gamma) \leq \nicefrac{|x-y|}{\varepsilon}$ such that we have for all $z \in \mathrm{tr}(\gamma)$ that
				\begin{equation}
					\dif(z, \partial N_{\delta} \cap O) \geq \frac{\varepsilon |z-x||z-y|}{|x-y|}. \label{eq: The geometric setup: Epsilon-Delta Condition}
				\end{equation}
				\item[(ii)] $O$ has positive radius near $N$: There is some $C > 0$ such that all connected components $O'$ of $O$ with $\partial O' \cap N \neq \varnothing$ satisfy $\diam (O') \geq C \delta$.
			\end{enumerate}

			\item[(Fat)]  $O^c$ is \textbf{locally $\boldsymbol{2}$-fat away from $\boldsymbol{N}$}.
		\end{enumerate}
	\end{assumptions}

	\begin{remark}
		Assumption (LU) is a quantitative local connectedness condition near the Neumann part that has been introduced in \cite[Def.~2.3]{Kato_Mixed} and is motivated from Jones' definition of an $(\varepsilon, \delta)$-domain \cite{Jones}. Assumption (Fat) imposes a lower bound on the relative $2$-capacity on small scales, see \cite[Sec.~3]{BCE-Gauss-for-MBC} for the precise definition. Eventually, it turns out \cite[Prop.~3.12]{BCE-Gauss-for-MBC} that (Fat) is under the mere assumption (LU) equivalent to a weak Poincaré inequality for boundary balls (P).
	\end{remark}

	\begin{remark}
		There is a common variation of (UITC) in which only radii $r \leq 1$ are considered. We refer to it as (ITC). It turns out that the constraint $r \leq 1$ in (ITC) can be changed to $r \leq r_0$ for any strictly positive $r_0$ up to a change of the implied constant $c$. Therefore, (UITC) implies (ITC). The converse is false in general, but holds if $O$ is bounded.   We are going to write (ITC) instead of (UITC) whenever we want to make explicit that $O$ is bounded.
		Observe that due to (UITC), $O$ is bounded if and only if $O$ has finite measure.
	\end{remark}

	We work with a concept of mean value free and locally constant functions that is adapted to $D$. Note that this terminology and notation is non-standard, but turns out to be very natural and convenient for our setting.

	\begin{definition}
		\label{def:locally_constant}
		We enumerate the bounded connected components $O'$ of $O$ that satisfy $\partial O' \cap D = \varnothing$ by $\{ O_m \}_m$ and write $\ISet$ for the countable index set.

		\begin{enumerate}
			\item Let $1 \leq p < \infty$. The \textbf{space $\boldsymbol{\L^p_0}$} consists of $f\in \L^p$ such that $\int_{O_m} f \, \d x = 0$ for every $m$.
			\item  We say that $f \in \L^1$ is \textbf{locally constant} if $f$ restricted to $O_m$ is constant for every $m$ and is zero otherwise. We identify this space with $\ell^1(\ISet)$.
		\end{enumerate}
	\end{definition}

	In the situation of Definition~\ref{def:locally_constant} put
	$\ON \coloneqq \bigcup_m O_m$ and $\OD \coloneqq O \setminus O_N$.
	We elaborate on the preceding definition by introducing a notion of \enquote{components}.

	\begin{definition}
		\label{def:components}
		The elements of $\Sigma \coloneqq \{ \OD, O_1, \dots, O_m \}$ are called \textbf{components} of $O$. Note that one has the disjoint decomposition $O = \bigcup_{O' \in \Sigma} O'$ by construction.
	\end{definition}

	\begin{remark}
		\label{rem:components}
		Some remarks are in order.
		\begin{enumerate}
			\item We stress that components need not be connected components. This is only ensured for the pure Neumann components $O_m$ and may fail for the component $\OD$. The components reflect a decoupling of the elliptic operator $L$ that stems from the presence of boundary conditions. We will see in Proposition~\ref{prop:kernel} why it is natural to have such a decoupling. In this sense, locally constant functions in the sense of Definition~\ref{def:locally_constant} are functions that are constant on decoupled parts of the operator $L$ and that respect themselves the Dirichlet boundary condition in $D$.
			\item Consider for instance the case where $\OD$ consists of two connected components, both equipped with pure Dirichlet boundary conditions. Large parts of our theory work if we consider the two connected components of $\OD$ as components of $O$. The notable exception is Lemma~\ref{lem:ITC_for_components} below. It is of course possible that $\OD$ satisfies (ITC) whereas this is false for all its connected components. To make our theory robust for such generalizations, we often write $D \cap \partial O'$, even though this intersection is always $D$ or $\varnothing$ with the choice of components from Definition~\ref{def:components}.
		\end{enumerate}
	\end{remark}

	\begin{lemma}
		\label{lem:ITC_for_components}
		Suppose that $O$ is bounded and satisfies $(\mathrm{ITC})$ and $(\mathrm{LU})$. Then the components of $O$ have a strictly positive distance to each other and satisfy $(\mathrm{ITC})$.
	\end{lemma}

	\begin{proof}
		For the first claim, let $O'$ and $O''$ be components of $O$. Assume without loss of generality that $O'$ is a pure Neumann component. If $\d(O', O'') < \nicefrac{\delta}{2}$, then there are $x \in O' \cap N_\delta$ and $y\in O'' \cap N_\delta$ with $|x-z| < \delta$. Then (LU) yields a path $\gamma \subseteq O$ that connects $x$ with $y$. Hence, $O' = O''$ since $O'$ is connected and the components are disjoint. Consequently, $\d(O', O'') \geq \delta$ for distinct components.

		We come to the second claim. Since $O$ is bounded, it suffices to check (ITC) for balls with radius $ r \leq r_0$ up to our disposal. We take $r_0 = \delta$.
		Let $O'$ be a component of $O$. Balls centered in $O'$ cannot hit other components of $O$ according to the first established claim from above. Therefore, (ITC) can be transferred from $O$ to $O'$.
	\end{proof}

	\begin{lemma}  \label{The geometric setup: Lemma: Assumption (D), upper estimate for measure}
		Assume $(D)$. Then there are $c > 0$ and $\eta \in (0,1)$ such that
		\begin{equation*}
			|\{ x\in O(x_0, r) : \d_D(x) < s \}| \leq c s^{\eta} r^{d - \eta}
		\end{equation*}
		holds true for all $x_0 \in O$, $r < \dO$ and $s > 0$.
	\end{lemma}

	\begin{proof}
		We can take $c = |\ball(0,1)|$ and any $\eta > 0$ for $r \leq s$. Now, assume that $s \leq r$. Since $\a \mapsto x^{\a}$ is decreasing on $(0, \infty)$ for $x \in (0, 1]$, we can ensure that $\eta \in (0,1)$ as soon as we have found some $\eta > 0$. Using that $D$ is uniformly porous, we can mimic the argument in \cite[Lem.~5.4.2 \& Lem.~A.1.8]{Bechtel_PhD} to prove the claim. Notice that we do not need the case distiction whether or not $r$ is at most $\nicefrac{1}{4}$.
	\end{proof}

	\subsection{Operator theoretic setup}  \label{Subsection: Operator theoretic setup}

	Here, we introduce the elliptic operator $L = - \Div (A \nabla \cdot)$ subject to mixed boundary conditions.

	To encode mixed boundary conditions, put $\rC^\infty_D \coloneqq \{ \j |_O : \j \in \smooth[\R^d \setminus D]\}$
	and define $\W_D^{1,2}$ as the closure of $\rC^\infty_D$ in $\W^{1,2}$.

	Let $A \in \L^{\infty}(O; \C^{d \times d})$. We realize the divergence form operator $L = - \Div (A \nabla \cdot )$ in $\L^2$ via the form
	\begin{equation*}
		a(u,v) \coloneqq \int_O A \nabla u \cdot \overline{\nabla v} \, \d x \qquad (u,v \in \W^{1,2}_D),
	\end{equation*}
	for which we impose a \textbf{homogeneous Gårding inequality}:
	\begin{align}
		\label{eq:garding}
			\exists \, \lambda > 0 \, \forall \, u \in \W_D^{1,2} \colon \quad \re a (u,u) \geq \lambda \Vert \nabla u \Vert_2^2.
	\end{align}
	Its domain is given by
	\begin{equation*}
		\D(L) \coloneqq \left\{ u \in \W_D^{1,2} : \exists \, Lu \in \L^2 \, \forall \, v \in \W_D^{1,2} : (Lu \, | \, v)_2 = a(u,v)  \right\}.
	\end{equation*}

	In general, $L$ is not injective. We characterize its null space in the following proposition. A similar observation was indicated in~\cite[p.~5]{Carbonaro-Dragicevic_Sg_Max_Operator}.

	\begin{proposition}
		\label{prop:kernel}
		Assume $(\mathrm{Fat})$ and $(\mathrm{LU})$. Then the space $\sN(L)$ consists exactly of those $f\in \L^2$ that are locally constant (in the sense of Definition~\ref{def:locally_constant}). Consequently, $\overline{\sR(L)} = \L^2_0$ and one has the topological splitting $\L^2 = \L^2_0 \oplus \sN(L)$.
		If $D$ is empty, the same result holds without the geometric assumptions.
	\end{proposition}

	\begin{proof}
		Let $u \in \sN(L)$. Using the ellipticity condition~\eqref{eq:garding} and the definition of $L$, we get
		\begin{align}
			\lambda \int_O |\nabla u|^2 \, \d x \leq \re a(u,u) \leq |(Lu \, | \, u)_2| = 0.
		\end{align}
		Therefore, $u$ is constant on each connected component of $O$. If $\OD \neq \varnothing$, let $x\in \partial \OD \cap D$. According to~\cite[Prop.~3.12]{BCE-Gauss-for-MBC}, there are $c,r>0$ depending on the geometry such that
		\begin{align}
			\int_{\OD(x,r)} |u|^2 \, \d x \lesssim \int_{O(x,cr)} |\nabla u|^2 \,  \d x = 0.
		\end{align}
		Therefore, $u$ is locally constant in the sense of Definition~\ref{def:locally_constant}. To the converse, let $u \in \W_D^{1,2}$ be locally constant. Then, $a(u,v) = 0$ for all $v \in \W^{1,2}_D$ and hence $u \in \sN(L)$.

		The remaining assertions follow readily from the well-known kernel-range splitting for sectorial operators.
	\end{proof}

	By Kato's form method \cite[Chap.~6]{Kato} this defines an m-accretive operator, which generates the holomorphic $\rC_0$-contraction semigroup $(\e^{-t^2 L})_{t \geq 0}$. Moreover, $L$ admits a maximal accretive square root $\sqrt{L}$ which generates the \textit{Poisson semigroup}
	$(P_t)_{t \geq 0} \coloneqq (\e^{- t \sqrt{L}})_{t \geq 0}$.

	\subsection{Property \boldmath$\mathrm{G(} \mu \mathrm{)}$}   \label{Subsection: Property G(mu)}

	We continue with the precise notion of kernel bounds that we will employ in order to prove our main result, Theorem~\ref{Main result and applications: Theorem: H_D^1=H_L^1}.

	\begin{definition}  \label{Property G(mu): Definition: G(mu)}
		Let $\mu \in (0,1]$. Say that $L$ has \textbf{property} $\boldsymbol{\mathrm{G}(\mu)}$ if the following holds:
		\begin{enumerate}
			\item[(G1)] For any $t > 0$ there is a measurable function $K_t \colon O \times O \to \C$ such that
			\begin{equation*}
				(\e^{-t^2 L} f)(x) = \int_O K_t(x,y) f(y) \dif y \qquad (f \in \L^2, \text{ a.e.} \; x \in O).
			\end{equation*}
			\item[(G2)] There are $b,c  > 0$ such that for each $0 < t < \dO$ one has
			\begin{equation*}
				|K_t(x,y)| \leq c t^{-d} \e^{- b \left(\frac{|x-y|}{t}\right)^2}.
			\end{equation*}
			\item[(G3)] There is $c>0$ such that for every $x,x',y,y' \in O$ and $0 < t < \dO$ one has
			\begin{equation*}
				|K_t(x,y) - K_t(x',y')| \leq c t^{-d - \mu} ( |x-x'| + |y-y'| )^{\mu}.
			\end{equation*}
		\end{enumerate}
	\end{definition}

	Let us mention that~$\mathrm{G}(\mu)$ originally goes back to~\cite{Auscher_Heat-Kernel}.

		\begin{remark} \label{Property G(mu): Remark: G(mu) stability}
		The following facts will be useful in this paper.
		\begin{enumerate}
			\item Property $\mathrm{G}(\mu)$ is stable under taking adjoints since the kernel of the adjoint semigroup is given by $K_t^*(x,y) = \overline{K_t(y,x)}$.

			\item Logarithmic convex combinations of (G2) and (G3) yield for all $\nu \in (0, \mu)$, each $x,y \in O$, $h \in \R^d$ with $y+h \in O$ and $0 < t < \dO$ the bound
			\begin{equation*}
				|K_t(x,y+h) - K_t(x,y)| \leq c t^{-d} \left( \frac{|h|}{t} \right)^{\nu} \e^{- b \left(\frac{|x-y|}{t}\right)^2}
			\end{equation*}
			provided that $|h| \leq \frac{|x-y|}{2}$, with different constants $b, c$ as before. A similar estimate holds true in the $x$-variable.

			\item For any $x, y \in O$ the function $z \mapsto K_z (x,y)$ defines an analytic map in the open sector $\{ z \in \C \setminus \{ 0 \} \colon |\arg(z)| < \nicefrac{\pi}{4} - \nicefrac{\omega_L}{2} \}$, where $\omega_L \in [0, \frac{\pi}{2})$ is the angle of sectoriality of $L$.
			Then a straightforward calculation shows that $\frac{1}{2} t^2 \partial_t K_t$ is the kernel associated with $- t^2 L \e^{-t^2 L}$ and satisfies the same kernel bounds as $K_t$ owing to Cauchy's theorem.
		\end{enumerate}
	\end{remark}

	The following lemma relates $(\mathrm{G}(\mu))$ with the null space of $L$.

	\begin{lemma}
		\label{lem:Gmu_kernel}
		Assume that $O$ is unbounded and that $L$ satisfies $(\mathrm{G}(\mu))$. Then $\sN(L)$ is trivial.
		In fact, we only use that $(\mathrm{G}(\mu))$ holds for $0 < t < \infty$, which is a consequence of $O$ unbounded.
	\end{lemma}

	\begin{proof}
		Let $u \in \ker(L)$. Standard theory for strongly continuous semigroups %
		gives
		\begin{align}
			\e^{-tL} u - u = \int_0^t \e^{-sL} L u \,\d s = 0, \quad t>0.
		\end{align}
		Therefore, $(\mathrm{G}(\mu))$ joint with Young's convolution inequality gives
		\begin{align}
			\| u \|_\infty = \| \e^{-t L} u \|_\infty \lesssim t^{-\frac{d}{2}} \| u \|_2 \to 0 \quad \text{as } t \to \infty.
		\end{align}
		Note that we use $\dO = \infty$ when taking the limit.
		Therefore indeed $u = 0$ as claimed.
	\end{proof}

	Surprisingly, $(\mathrm{G}(\mu))$ also imposes the following crucial geometrical constraint in the case of mixed boundary conditions.

	\begin{corollary}
		\label{cor:ON_implies_bounded}
		Assume $(\mathrm{Fat})$, $(\mathrm{LU})$ and $(\mathrm{G}(\mu))$. Then $\ON \neq \varnothing$ implies that $O$ is bounded.
	\end{corollary}

	\begin{proof}
		Assume for the sake of contradicton that $\ON \neq \varnothing$, but $O$ is unbounded. By Proposition~\ref{prop:kernel}, $\sN(L)$ consists of the locally constant functions. Therefore, $\sN(L)$ is non-trivial by hypothesis. On the other hand, Lemma~\ref{lem:Gmu_kernel} asserts that $\sN(L)$ is trivial when $O$ is unbounded, a contradiction.
	\end{proof}

	The following fact was also stated in~\cite{Auscher-Russ}. We present a proof based on the previous two results. A more direct proof could be given with the conservation property.

	\begin{corollary}
		\label{cor:NBC_Gmu_unbounded}
		Assume that $L$ is subject to pure Neumann boundary conditions and satisfies $(\mathrm{G}(\mu))$ for $0<t<\infty$. Then $O$ is unbounded.
	\end{corollary}

	\begin{proof}
		According to Lemma~\ref{lem:Gmu_kernel}, $\sN(L)$ is trivial. Assume to the contrary that $O$ is bounded. Then all its connected components are bounded and Proposition~\ref{prop:kernel} yields that $\sN(L)$ consists of the locally constant functions on $O$, a contradiction.
	\end{proof}

	Next, we introduce the notion of $\L^2 - \Cdot^{\nu}$-boundedness. To do so, we put
	\begin{equation*}
		\Vert u \Vert_{\Cdot^{\nu}} \coloneqq \sup_{\substack{x,y \in O, \\ x \neq y}} \frac{|u(x) - u(y)|}{|x-y|^{\nu}}.
	\end{equation*}

	\begin{definition} \label{Property G(mu): Definition: Off-diagonal estimates}
		Let $1 \leq p \leq \infty$, $\nu \in (0,1]$ and $\mathcal{T} \coloneqq (T(z))_{z\in U} \sub \mathcal{L}(\L^2)$, where $U \subseteq \C \setminus \{ 0 \}$. Say that $\mathcal{T}$ is \textbf{$\boldsymbol{\L^p - \Cdot^{\nu}}$-bounded} if $T(z) f \in \Cdot^{\nu}$ for all $z \in U$ and $f \in \L^p \cap \L^2$, with the estimate
		\begin{equation*}
			\Vert T(z) f \Vert_{\Cdot^{\nu}} \les |z|^{- \nu - \frac{d}{p}} \Vert f \Vert_p.
		\end{equation*}
	\end{definition}

	The following lemma can be found in~\cite[Cor.~5.11]{BCE-Gauss-for-MBC}. They also state a converse statement which we will not need in the course of this article.

	\begin{lemma}
		\label{lem:Hoelder_boundedness}
		If $L$ satisfies $(\mathrm{G}(\mu))$, then $( \e^{-t^2 L} )_{0 < t < \d(O)}$ is $\L^2 - \Cdot^\nu$-bounded for every $\nu \in (0,\mu)$.
	\end{lemma}

	Now, \cite[Cor.~5.15]{BCE-Gauss-for-MBC} can be used to prove the following.

	\begin{lemma} \label{Property G(mu): Lemma: Poisson kernel}
	Assume that $L$ has property $\mathrm{G}(\mu)$. Then the subordination formula for the Poisson semigroup \cite[p.~169]{Fattorini-Subordination}
	\begin{equation*}
		\e^{- t \sqrt{L}} f = \frac{1}{\sqrt{\pi}} \int_0^{\infty}  \e^{-u} \e^{- \frac{t^2}{4 u} L} f \, \frac{\d u}{\sqrt{u}}  \qquad (t > 0, f \in \L^2)
	\end{equation*}
	implies that its kernel $p_t$ is given by the $\mu$-Hölder continuous function
	\begin{equation}
		p_t(x,y) = \frac{1}{\sqrt{\pi}} \int_0^{\infty} K_{\frac{t}{2 \sqrt{u}}}(x,y) \e^{-u} \, \frac{\d u}{\sqrt{u}},   \label{eq: Property G(mu): Subordination for kernel}
	\end{equation}
	that satisfies the estimates
	\begin{equation} \label{eq: Kernel bound for Poisson semigroup}
		|p_t(x,y)| \les t^{-d} \left( 1 + \frac{|x-y|}{t} \right)^{-(d+1)}  \quad (0 < t < \dO, x,y \in O),
	\end{equation}
	and
	\begin{equation} \label{eq: Hoelder Kernel bound for Poisson semigroup}
		|p_t(x,y) - p_t(x', y')| \les t^{-d - \mu} ( |x-x'| + |y-y'| )^{\mu} \quad (0 < t < \dO, x,x', y, y' \in O).
	\end{equation}
	\end{lemma}

	\begin{remark} \label{Property G(mu): Remark: Heat kernel inherits properties to poisson kernel}
		Similar observations as in Remark~\ref{Property G(mu): Remark: G(mu) stability} for the heat semigroup can be made for the Poisson semigroup, too. %
	\end{remark}

	\begin{lemma}  \label{Preliminary results: Lemma: Linfty-extension of the semigroup}
		Let $0 < t < \dO$ and let either $\j(z) = \sqrt{z} \e^{-\sqrt{z}}$ or $\j(z) = z \e^{-z}$.
		Then there is a kernel $q_t \colon O \times O \to \C$ such that the operator $f \mapsto \int_O q_t(\cdot,y) f(y) \, \d y$ is bounded on $\L^p$ for all $1\leq p \leq \infty$ and one has the identity
		\begin{equation*}
			\j(t^2 L) f = \int_O q_t (\cdot,y) f(y) \, \d y \qquad (f \in \L^p \cap \L^2).
		\end{equation*}
	\end{lemma}

	\begin{proof}
		Consider $\j(z) = z \e^{-z}$ first. Then $\j(t^2 L) = t^2 L \e^{-t^2 L}$ and its kernel is given by $\frac{1}{2} t^2 \partial_t K_t$ and satisfies the bound (G2) from Definition~\ref{Property G(mu): Definition: G(mu)}, see Remark~\ref{Property G(mu): Remark: G(mu) stability}.
		Now the $\L^p$-estimate
		\begin{align}
			\Vert \j(t^2 L) f \Vert_p \les \Vert f \Vert_p
		\end{align}
		follows from the kernel bounds by Young's convolution inequality.
		Consistency follows by construction, compare also with (G1).

		The argument for $\j(z) = \sqrt{z} \e^{-z}$ is similar. Indeed, the kernel bounds can be obtained in the same way, see Lemma~\ref{Property G(mu): Lemma: Poisson kernel} and Remark~\ref{Property G(mu): Remark: Heat kernel inherits properties to poisson kernel}.
		Note that $q_t$ satisfies only the Poisson bound~\eqref{eq: Kernel bound for Poisson semigroup} in this case, which is nevertheless sufficient to invoke Young's convolution inequality.
	\end{proof}

	\begin{lemma}  \label{The geometric setup: Lemma: Kernel vanishes in D}
		Assume $(\mathrm{Fat})$, $(\mathrm{LU})$ and $(\mathrm{G}(\mu))$, and let again either $\j(z) = \sqrt{z} \e^{-\sqrt{z}}$ or $\j(z) = z \e^{-z}$. Then the kernel $q_t$ from Lemma~\ref{Preliminary results: Lemma: Linfty-extension of the semigroup} satisfies $q_t (x,y) = 0$ for all $x \in D$ and $y \in O$.
	\end{lemma}

	\begin{proof}
		As before, consider first $\j(z) = z \e^{-z}$.
		Let $f \in \L^2$ and $0 < t < \dO$. One the one hand, Lemma~\ref{lem:Hoelder_boundedness} yields $\j(t^2 L) f \in \Cdot^{\mu}$. On the other hand, $\j(t^2 L) f \in \D(L) \subseteq \W_D^{1,2}$. Hence, $\j(t^2 L) f \in \W_D^{1,2} \cap \Cdot^{\mu}$. Now,
		\cite[Lem.~4.8]{BCE-Gauss-for-MBC} implies that
		\begin{equation*}
			\int_O  q_t (x,y) f(y) \, \d y = \j(t^2 L) f(x) = 0, \qquad x \in D.
		\end{equation*}
		Recall that $q_t(x, \cdot)$ is continuous. Thus, we infer $q_t (x,y) = 0$ for all $y \in O$ by replacing $f$ with an approximation of the identity in $y$.

		Now if $\j(z) = \sqrt{z} \e^{-\sqrt{z}}$, then the claim follows from the first case in conjunction with~\eqref{eq: Property G(mu): Subordination for kernel} and Remark~\ref{Property G(mu): Remark: Heat kernel inherits properties to poisson kernel}.
	\end{proof}

	The conclusions of Lemmas~\ref{Preliminary results: Lemma: Linfty-extension of the semigroup} and~\ref{The geometric setup: Lemma: Kernel vanishes in D} are the central properties of the operators $\j(t^2 L)$. Thus, we can usually ignore the precise choice of $\j$, which we manifest in the following definition.

	\begin{definition}[kernel $q_t$]
		\label{def:qt}
		In both cases $\j(z) = \sqrt{z} \e^{-\sqrt{z}}$ or $\j(z) = z \e^{-z}$ simply write $q_t$ for the kernel of the operator $\j(t^2 L)$.
	\end{definition}

	\section{Review of Hardy spaces on open sets} \label{Section: Review of Hardy spaces on open sets}

	There appear different notions of Hardy spaces on an open set $O$ in the literature. On the one hand, there are several atomic descriptions, notably those by Coifman--Weiss~\cite{Coifman-Weiss} and Miyachi~\cite{Miyachi-Atoms}, and on the other hand also \enquote{global} constructions in which $\H^1(\R^d)$ is the point of departure to define new spaces as subspaces or quotient spaces~\cite{CKS}.
	It turns out that all these spaces can be sorted into two categories: spaces that are associated with an operator subject to pure Dirichlet or pure Neumann boundary conditions. \newline To find a reasonable definition for the case of mixed boundary conditions, it is intrusive to understand these constructions better. Also, our very general geometric framework allows us to sharpen results on the pure Dirichlet and Neumann cases in the literature.

	\subsection{Atomic spaces on an open set} \label{subsec:atomic_open}

	All atomic $\H^1$-spaces consist of functions $f$ that can be represented by an \textbf{atomic representation} $f = \sum_j \lambda_j a_j$, where $(\lambda_j)_j \in \ell^1$ and the $a_j$ are \enquote{atoms}. The decisive question here is the definition of an atom. In the case of the space $\H^1_\CW$ introduced by Coifman--Weiss~\cite{Coifman-Weiss}, a function $a \colon O \to \C$ is called \textbf{atom} if there is a ball $\ball \subseteq \R^d$ centered in $O$ such that
	\begin{align}
		(\mathrm{i}) \; \supp(a) \subseteq \ball \cap \overline{O}, \quad (\mathrm{ii}) \; \| a \|_2 \leq |\ball \cap O|^{-\frac{1}{2}} \quad \& \quad (\mathrm{iii}) \int_O a \; \d x = 0.
	\end{align}
	We refer to the properties (i), (ii) and (iii) as \emph{localization}, \emph{size condition} and \emph{cancellation condition}. Here, the size condition is formulated with $\L^2$. In both the Coifman--Weiss and Miyachi cases, a more powerful result with a size condition in $\L^\infty$ is feasible, but for the scope of our paper the $\L^2$-perspective is the right one.
	Observe that for $O = \R^d$ we recover $\H^1_\CW(\R^d) = \H^1(\R^d)$.
	Let us remark that in the special case where $O$ is of finite measure, constant functions are also members of $\H^1_\CW$ even though they are not mean value free.

	For the space $\H^1_\Miyachi$ by Miyachi~\cite{Miyachi-Atoms}, the notion of an atom is relaxed. %
	First, if $A \colon \R^d \to \C$ is a function and $\ball \subseteq \R^d$ a ball centered in $O$ such that the localization, size and cancellation conditions are fulfilled, then $a=A|_O$ is an atom. In particular, all Coifman--Weiss atoms are also Miyachi atoms under (UITC).
	Second, if $a$ is a function and $\ball$ is a ball centered in $O$ that fulfill the localization and size conditions, and $2\ball \subseteq O$ but $5\ball \cap O^c \neq \varnothing$, then $a$ is also a Miyachi atom.

	There appears another variation of Miyachi atoms in the literature, employed for instance in~\cite[Def.~4.3]{JFA-Hardy} and implicitly used in~\cite{Auscher-Russ}. Given $a$ and $\ball$ satisfying the localization and size condition as before, they either demand that $2\ball \subseteq O$ and $4\ball \cap \partial O \neq \varnothing$, or that $4\ball \subseteq O$ and $a$ satisfies the cancellation condition.

	\begin{proposition}[both Miyachi spaces coincide]
		\label{prop:Miyachi_coincide1}
		The Miyachi spaces in the sense of~\cite{Miyachi-Atoms} coincide with those in the sense of~\cite{JFA-Hardy}.
	\end{proposition}

	\begin{proof}
		Up to a change of implied constants, it is negligible that one definition is with respect to $4\ball$ and the other one with respect to $5\ball$. In particular, the Miyachi space in the sense of~\cite{JFA-Hardy} embeds into the space in the sense of~\cite{Miyachi-Atoms}.

		For the converse inclusion, let $a=A|_O$ be an atom verifying the cancellation condition and such that either $2 \ball \subseteq O$ or $5 \ball \cap O^c \neq \varnothing$. In the first case, $a$ is immediately an atom in the sense of~\cite{JFA-Hardy}. Otherwise, we can decompose $a$ into a sum of atoms without a cancellation condition, for instance using a localization with the Vitali covering lemma. The details are presented in~\cite[Proof of Lem.~4.7]{JFA-Hardy}.
		In a nutshell, we start with the trivial covering $$\ball \subseteq \bigcup_{x\in \ball} \ball(x, c\delta(x)),$$
		where $\delta(x)$ is the distance from $x$ to $\partial O$ and $c = \nicefrac{1}{10}$ is a scaling constant. Vitali's covering lemma yields a countable subcollection $\{ x_j \}_j$ such that the balls $\ball_j \coloneqq \ball(x_j, 5 c \delta(x_j))$ are essentially disjoint
		and they satisfy $2 \ball_j \subseteq \R^d \setminus \partial O$ as well as $4 \ball_j \cap \partial O \neq \varnothing$ by construction. The decomposition of $a$ is then defined by localization in virtue of indicator functions to $\ball_j$.

		We remark that if the spaces were defined using cubes instead of balls, this decomposition could be achieved using a local Whitney decomposition relative to $2Q$. This alternative argument is employed in~\cite[p.\,167]{Auscher-Russ}.
	\end{proof}

	The following remark explains why neither of the two definitions is right away suitable to define spaces adapted to mixed boundary conditions.

	\begin{remark}
		\label{rem:miyachi}
		The Miyachi space in the sense of~\cite{JFA-Hardy} has a priori a better localization of atoms. This is needed in the identification of $L$-adapted spaces, which becomes apparent from the references given in the proof. A naive idea to treat mixed boundary conditions might be to simply replace the distance function from $x$ to $\partial O$ by the distance function from $x$ to $D$. Unfortunately, this does not work. The reason is that balls defined with respect to $\partial O$ are either contained completely inside $O$ or completely outside of $O$. Balls completely outside $O$ can be \enquote{thrown away}. This is not true anymore in the case of mixed boundary conditions.
	\end{remark}

	The definition that we are going to give in Section~\ref{Section: Duality theory for D-adapted spaces} is a third alternative way to define the Miyachi space $\H^1_\Miyachi$ as becomes apparent from Proposition~\ref{prop:atomic_pure_cases}.

	\subsection{Constructions based on the Euclidean Hardy space}
	\label{subsec:global_constructions}

	The first construction is by means of restriction. Let $|_O$ be the pointwise restriction of functions on $\R^d$ to $O$. The \textbf{restriction Hardy space} $\H^1_\restrict$ is defined by
	\begin{align}
		\H^1_\restrict \coloneqq \bigl\{ f \colon f = F|_O, \; F\in \H^1(\R^d) \bigr\}, \qquad \| f \|_{\H^1_\restrict} \coloneqq \inf \| F \|_{\H^1(\R^d)},
	\end{align}
	where the infimum is taken over all $F\in \H^1(\R^d)$ with $f = F|_O$. Under suitable geometric assumptions, $\H^1_\restrict$ corresponds to the Miyachi Hardy space. For instance, using that $O$ is uniform in a certain sense, $\H^1_\restrict = \H^1_\Miyachi$ follows from~\cite[Sec.~6]{Miyachi-Atoms}. Another approach uses thickness of $O^c$. The main argument was already employed in~\cite[p.\,304]{CKS}, even though the formulation was less general.
	\begin{proposition}
		\label{prop:restrict}
		Assume that $O$ is uniformly exterior thick, that is to say, there exists a constant $c>0$ such that for any ball $\ball$ centered in $\partial O$ one has $|\ball \cap O^c| \geq c |\ball|$. Then $\H^1_\restrict = \H^1_{\mathrm{Mi}}$.
	\end{proposition}
	\begin{proof}
		The inclusion $\H^1_\restrict \subseteq \H^1_\Miyachi$ is explained in~\cite[Proof of Thm.~2.7]{CKS} without imposing any geometric assumptions. We mention that this inclusion reuses the argument presented in the proof of Proposition~\ref{prop:Miyachi_coincide1}.

		For the converse inclusion, we only have to consider atoms $a$ supported in a ball $\ball$ such that $2\ball \subseteq O$ but $5\ball \cap O^c \neq \varnothing$, because the other atoms are by definition restrictions of atoms for $\H^1(\R^d)$. The idea is to consider $a$ as an atom localized in the enlarged ball $6\ball$. Thanks to the exterior thickness condition, we can ensure the cancellation condition by modifying $a$ on $O^c  \cap \ball(x,r)$, where $x \in 5\ball \cap \partial O$ and $r=r(\ball)$.
	\end{proof}

	The result already appeared on Lipschitz domains as a sideproduct in~\cite{Auscher-Russ}. The following example shows that we can go even below the class of Lipschitz regular domains.

	\begin{example}
		Let $\bigstar \subseteq \R^2$ denote the interior of the von Koch snowflake depicted in Figure~\ref{fig:snowflake} below. Then one has $$\H^1_\restrict(\bigstar) = \H^1_\Miyachi(\bigstar).$$ Indeed, the assumptions of Proposition~\ref{prop:restrict} are satisfied, see~\cite[p.~349]{MMM}.
	\end{example}

	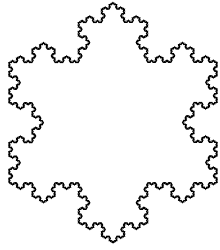
\begin{figure}
		\label{fig:snowflake}
		\begin{tikzpicture}[scale=0.8,decoration=Koch snowflake]
			\draw[] decorate{decorate{decorate{decorate{decorate{ (-1.732,0) -- (0,3) -- (1.732, 0) -- (-1.732, 0)}}}}} ;
		\end{tikzpicture}
		\caption{Sketch of a von Koch snowflake.}
	\end{figure}

	The second construction is based on zero extension. Write $\E_0$ for the zero extension operator. The \textbf{zero extension Hardy space} $\H^1_\zero$ is defined by
	\begin{align}
		\H^1_\zero \coloneqq \bigl\{ f \colon \E_0 f \in \H^1(\R^d) \bigr\}, \qquad \| f \|_{\H^1_\zero} \coloneqq \| \E_0 f \|_{\H^1(\R^d)}.
	\end{align}
	It is clear that $\H^1_\CW \subseteq \H^1_\zero$. The converse inclusion is less trivial. Using operator adapted spaces, $\H^1_\zero = \H^1_\CW$ was shown in~\cite{Auscher-Russ}. We are going to sharpen this result in our less restrictive geometric framework.

	\begin{theorem}
		\label{thm:H1z}
		Suppose that $O\subseteq \R^d$ is open and satisfies $(\mathrm{UITC})$ and $(\mathrm{LU})$. Then one has $\H^1_\zero = \H^1_\CW$. If $O$ is bounded, then already $(\mathrm{LU})$ is sufficient.
	\end{theorem}
	We postpone the proof to the end of Section~\ref{Section: Inclusion of H_D^1(O) into Operator-adapted space}.

	\subsection{Dirichlet vs Neumann Hardy spaces}
	\label{subsec:dir_vs_neumann}

	Write $-\Delta_0$ and $-\Delta_N$ for the Dirichlet and Neumann Laplacians on $O$. If $O$ is a Lipschitz domain, then the following classification was derived in~\cite{Auscher-Russ}:

	\setlength{\arrayrulewidth}{0.3mm}
	\setlength{\tabcolsep}{18pt}
	\renewcommand{\arraystretch}{1.5}

	\begin{center}
		\begin{tabular}{ c | c }
			Dirichlet & Neumann \\
			\hline
			$\H^1_{-\Delta_0}$ & $\H^1_{-\Delta_N}$ \\
			$\H^1_\Miyachi$ & $\H^1_\CW$ \\
			$\H^1_\restrict$ &  $\H^1_\zero$
		\end{tabular}
	\end{center}

	Our main result (Theorem~\ref{Main result and applications: Theorem: H_D^1=H_L^1}) can be interpreted as adding a third column that corresponds to mixed boundary conditions to this classification. The third line, a global characterization, has yet to remain empty in the case of mixed boundary conditions, but gives rise to subsequent research.

\section{Duality theory for $D$-adapted spaces}  \label{Section: Duality theory for D-adapted spaces}

We introduce atomic Hardy spaces $\H^1_D$ subject to a cancellation condition in $D$ and provide duality relations with corresponding spaces $\VMO_D$ and $\BMO_D$.

\subsection{Philosophy behind our atoms}   \label{Subsection: Philosophy behind the atoms}

So far we have seen that classical atoms in the sense of Coifman--Weiss correspond to elliptic operators subject to pure Neumann boundary conditions and atoms in the sense of Miyachi to elliptic operators subject to pure Dirichlet boundary conditions. While classical atoms have a cancellation condition up to the boundary, for Miyachi atoms this is not the case anymore. Instead, the boundary conditions are supposed to compensate for a lack of cancellation of the atoms. Indeed, our Lemma~\ref{The geometric setup: Lemma: Kernel vanishes in D} applied to the pure Dirichlet case gives that the kernel of the operator $\j(t^2 L)$ vanishes in $\partial O$. These operators will be the starting point to define operator-adapted spaces in Section~\ref{Section: Identification of operator-adapted and atomic Hardy spaces} and we will see, for instance in the proof of the central Lemma~\ref{Inclusion of H_D^1(O) into Operator-adapted space: Lemma: Interpolation bound}, that the cancellation of atoms can often be replaced by a vanishing kernel in a straightforward way.

In the realm of mixed boundary conditions, Lemma~\ref{The geometric setup: Lemma: Kernel vanishes in D} shows that the kernel of $\j(t^2 L)$ vanishes in $D$. Consequently, atoms for a space $\H^1_D$ should have the cancellation condition provided they are not close to $D$. We have already discussed in Remark~\ref{rem:miyachi} that it is not sufficient to simply replace $\partial O$ in the definition of a Miyachi atom by $D$. Instead, we provide another variation of Miyachi's definition that works in the situation of mixed boundary conditions and at the same time allows to treat the pure Dirichlet and Neumann cases in a unified way.

\subsection{Precise definitions}   \label{Subsection: Precise definitions}
Let $x \in O$ and $O'$ the component of $O$ that contains $x$. Write $B(x,r) = \ball(x,r) \cap O'$ for the localization of $\ball(x,r)$ to its component. More generally, if $\ball$ is a ball centered in a component $O'$, write $B = \ball \cap O'$.

\begin{definition}
	\label{def:balls}
	Any ball $B = B(x,r)$ with $x \in O$ is called \textbf{usual}. If $x \in O'$ for $O' \in \Sigma$ and $ \ball(x, 2r) \cap (D \cap \partial O') \neq \varnothing$, then say that $B$ is \textbf{near $\boldsymbol{D}$}.
\end{definition}

\begin{definition}
	\label{def:H1D_atoms}
	Let $B$ be a (usual) ball and $a \colon O \to \C$ be measurable.

	\begin{enumerate}

		\item We call $a$ a \textbf{usual $\boldsymbol{\H_D^1}$-atom}, if
		\begin{equation*}
			\supp(a) \sub \overline{B}, \quad \Vert a \Vert_2 \leq |B|^{- \frac{1}{2}} \quad \& \quad \int_O a \, \d x = 0.
		\end{equation*}

		\item We call $a$ an  \textbf{$\boldsymbol{\H_D^1}$-atom near $\boldsymbol{D}$}, if $B$ is near $D$,
		\begin{equation*}
			\supp(a) \sub \overline{B} \quad \& \quad \Vert a \Vert_2 \leq |B|^{- \frac{1}{2}}.
		\end{equation*}

	\end{enumerate}
\end{definition}

\begin{definition}
	\label{Definition: H_D^1}
	The \textbf{atomic Hardy space $\boldsymbol{\H_D^1}$} consists of all measurable functions $f \colon O \to \C$ for which there is a sequence $(\lambda_j)_j \in \ell^1$ and a sequence of $\H^1_D$-atoms $(a_j)_j$ such that $f = \sum_j \lambda_j a_j$ pointwise almost everywhere. We equip this space with the norm
	\begin{equation*}
		\Vert f \Vert_{\H_D^1} \coloneqq \inf \Vert (\lambda_j)_j \Vert_{\ell^1},
	\end{equation*}
	where the infimum is taken over all atomic decompositions of $f$.

	In addition, we denote by $\boldsymbol{\H_{D,a}^1}$ the subspace of all \textbf{finite linear combinations of $\boldsymbol{\H_D^1}$-atoms}. This space is by construction dense in $\H_D^1$.
\end{definition}

\begin{remark} \label{rem:atomic_convergence}
	Note that every atomic decomposition also converges in $\L^1$ and $\Vert f \Vert_1 \leq \Vert f \Vert_{\H_D^1}$ for all $f \in \H_D^1$. This means that the inclusion map $\H_D^1 \sub \L^1$ is a contraction.
\end{remark}

The following proposition shows that Definition~\ref{def:H1D_atoms} indeed extends and unifies the pure Dirichlet and Neumann cases in our geometric setting.

\begin{proposition}[compatibility with the Coifman--Weiss and Miyachi spaces]
	\label{prop:atomic_pure_cases}
	Suppose (UITC).
	In the pure Neumann case $D = \varnothing$ and if the component have positive distance to each other (compare with Lemma~\ref{lem:ITC_for_components}) one finds $\H^1_\varnothing \oplus \ell^1(\ISet) = \H^1_\CW$. In the pure Dirichlet case $D = \partial O$ we recover $\H^1_{\partial O} = \H^1_\Miyachi$.
\end{proposition}

\begin{proof}
	We show both cases separately.

	\textbf{pure Neumann case}.
	The main difference between $\H^1_\varnothing$ and $\H^1_\CW$ is the concept of components in the definition of the former space.
	Indeed, if $O$ is unbounded and hence has infinite measure, $\ISet = \varnothing$ and all $\H^1_\varnothing$-atoms are usual and coincide with the Coifman--Weiss atoms. Therefore, $\H^1_\varnothing \oplus \ell^1(\ISet) = \H^1_\varnothing = \H^1_\CW$. Otherwise, if $O$ is bounded, constants are added to $\H^1_\CW$ and $\ISet \neq \varnothing$.

	Start with the inclusion $\H^1_\varnothing \oplus \ell^1(\ISet) \subseteq \H^1_\CW \oplus \C$.
	All $\H^1_\varnothing$-atoms are usual, so in particular contained in $\H^1_\CW$.
	Fix some $m\in \ISet$ and $c_m \in \C$. We construct a function $f \in \H^1_\CW \oplus \C$ that coincides with $c_m$ on $O_m$ and is zero otherwise as follows: let $h \in \C$ to our disposal. The function $a$ given by $h$ on $O_m$ and $-h \frac{|O_m|}{|O \setminus O_m|}$ on $O\setminus O_m$ is (up to normalization) a Coifman--Weiss atom. Put $c \coloneqq h \frac{|O_m|}{|O \setminus O_m|}$. Then $a + c$ vanishes on $O \setminus O_m$ and takes the value $h(1 + \frac{|O_m|}{|O \setminus O_m|})$ on $O_m$. An appropriate choice of $h$ gives $h(1 + \frac{|O_m|}{|O \setminus O_m|}) = c_m$ as claimed. Iterating this process gives
	the first inclusion.

	Next, we show $\H^1_\CW \oplus \C \subseteq \H^1_\varnothing \oplus \ell^1(\ISet)$. It is clear that a constant function is also locally constant. It remains to show the inclusion for an atom $a\in \H^1_\CW$.
	The discrepancy between a Coifman--Weiss atom and a usual atom for $\H^1_\varnothing$ is that the localization condition for Coifman--Weiss atoms ignores the concept of components, the same is true for the cancellation condition. However, if $a$ is associated with a ball $\ball(x,r) \cap O$, and $r$ is at most the distance between the components, then $\ball(x,r) \cap O = B(x,r)$ and the localization and cancellation conditions for a usual $\H^1_\varnothing$-atom are verified. Otherwise, $|\ball(x,r) \cap O| \simeq |O_m|$ for any $m$. Put $c = \sum_m (a)_{O_m} \1_{O_m}$ and write $a = \sum_m \1_{O_m} (a - (a)_{O_m}) + c$. The functions $\1_{O_m} (a - (a)_{O_m})$ are constant multiples of $\H^1_\varnothing$-atoms associated with $O_m$ and $c$ is locally constant, hence the second inclusion is complete.

	\textbf{pure Dirichlet case}. Observe that we could equivalently define $\H^1_{\partial O}$-atoms near $\partial O$ by the condition $5\ball \cap \partial O \neq \varnothing$. We use this throughout this proof. Now recall from Section~\ref{Section: Review of Hardy spaces on open sets} that the space $\H^1_\Miyachi$ consists of two types of atoms. (i) $a=A|_O$ is the restriction of an $\H^1(\R^d)$-atom associated with a ball $\ball$ centered in $O$ or (ii) $a$ is associated with a ball $\ball$ that satisfies $2\ball \subseteq O$, $5\ball \cap O^c \neq \varnothing$ and $a$ satisfies the localization and size conditions.
	We start with the inclusion $\H^1_\Miyachi \subseteq \H^1_{\partial O}$. Let $a$ be of type (i). Then either $2\ball \subseteq O$, in which case $a$ is a usual atom for $\H^1_{\partial O}$, or $2\ball \cap \partial O \neq \varnothing$, in which case $a$ is an $\H^1_{\partial O}$-atom near $\partial O$. Next, let $a$ be an atom of type (ii). Then automatically $a$ is an $\H^1_{\partial O}$-atom near $\partial O$.

	For the converse inclusion $\H^1_{\partial O} \subseteq \H^1_\Miyachi$, let $a$ be an $\H^1_{\partial O}$-atom near $\partial O$ associated with $B$. If $2\ball \subseteq O$, then $a$ is immediately an atom of type (ii). Otherwise, employ the localization argument used in the proof of Proposition~\ref{prop:Miyachi_coincide1} to conclude. If $a$ is a usual $\H^1_{\partial O}$-atom, then extend $a$ from $B$ to $\ball$ by zero. Write $A$ for this function on $\ball$. By (UITC), $A$ is (up to normalization) an $\H^1(\R^d)$ atom, so $a=A|_O$ is of type (i).
\end{proof}

The space $\L^2_{\loc}$ in the following definition consists of all measurable functions that are square integrable on any bounded subset of $O$.

\begin{definition}
	\label{Definition: BMO_D}
	The \textbf{space of bounded mean oscillation} $\boldsymbol{\BMO_D}$ consists of all $f\in \L_{\loc}^2$ for which the (semi-)norm
	\begin{align*}
		\Vert f \Vert_{\BMO_D} \coloneqq \sup_{B \, \text{usual}} \left( \fint_{B} |f - (f)_{B}|^2 \, \d x \right)^{\frac{1}{2}} \; \lor \; \sup_{B \, \text{near} \, D} \left( \fint_{B} |f|^2 \, \d x \right)^{\frac{1}{2}}
	\end{align*}
	is finite.
\end{definition}

The definition of $\BMO_D$ already indicates the natural $\H_D^1 - \BMO$-duality. The key observation is the following lemma.

\begin{lemma} \label{Precise Definitions: Lemma: H1-BMO-pairing}
	It holds
	\begin{equation}
		\left| \int_O f g \, \d x \right| \leq \Vert f \Vert_{\H_D^1} \Vert g \Vert_{\BMO_D} \qquad (f \in \H_{D}^1, g \in \L^{\infty}). \label{eq: H1-BMO pairing}
	\end{equation}
\end{lemma}

\begin{proof}
	Let $f = \sum_{j} \lambda_j a_j$ be an atomic decomposition of $f \in \H_D^1$. Since $g \in \L^{\infty}$ and the series converges in $\L^1$ we obtain
	\begin{equation*}
		 \int_O f g \, \d x = \sum_j \lambda_j \int_{B_j} a_j g \, \d x.
	\end{equation*}
	If $a_j$ is near $D$, then we directly apply Hölder's inequality on $\int_{B_j} a_j g \, \d x$. In the other case we exploit the moment condition of $a_j$ on $B_j$, which allows us to subtract $(g)_{B_j}$ from $g$. Afterwards, apply Hölder's inequality. Hence, using the size condition for the atoms and the very definition of the $\BMO_D$-norm we get
	\begin{equation*}
		 \left| \int_O f g \, \d x \right| \leq \Vert (\lambda_j)_j \Vert_{\ell^1} \Vert g \Vert_{\BMO_D}.
	\end{equation*}
We conclude by taking the infimum over all possible atomic decompositions of $f$.
\end{proof}

\begin{remark}[direct decomposition of $\H^1_D$ and $\BMO_D$]
	\label{rem:atomic_direct_sum}
	The localization property of $\H^1_D$-atoms and its counterpart in the $\BMO_D$-norm give rise to the topological splittings
	\begin{equation*}
		\H_D^1 = \bigoplus_{O' \in \Sigma} \H_{D \cap \partial O'}^1(O') \qquad \text{and} \qquad \BMO_D = \bigoplus_{O' \in \Sigma} \BMO_{D \cap \partial O'}(O').
	\end{equation*}
\end{remark}

As a consequence of the preceding remark we are going to impose the following assumption for the rest of this section.

\begin{assumption}
	\label{ass:atomic}
	Suppose that $O$ satisfies (UITC) and consists of only one component.
\end{assumption}

It follows from Remark~\ref{rem:atomic_direct_sum} that all results from this section stay true if $O$ consists of multiple components that all satisfy (UITC). According to Corollary~\ref{cor:ON_implies_bounded} and Lemma~\ref{lem:ITC_for_components}, this is a consequence of $(\mathrm{G}(\mu))$ in conjunction with (Fat) and (LU), so we can safely use all results developed in this section in the setting
of Theorem~\ref{Main result and applications: Theorem: H_D^1=H_L^1}.

\subsection{\boldmath$\H_D^1- \BMO_D$-duality} \label{Subsection: H1-BMO-duality}

In this part we extend the classical $\H^1 - \BMO$-duality of Fefferman \cite{Fefferman-BMO-H1-duality}. This comprises also the pure Neumann case established by Coifman and Weiss \cite{Coifman-Weiss} and we stress that no geometric assumptions are used.  %

\begin{theorem} \label{H1-BMO-duality: Theorem: Main result}
	The space $\BMO_D$ is the dual space of $\H_D^1$ in the following sense:

	\begin{enumerate}
		\item Each $g \in \BMO_D$ extends to a bounded linear functional on $\H_D^1$ via $\ell_g(f) = \int_O g f \, \d x$ for each $f \in \H_{D, a}^1$ and $\Vert \ell_g \Vert_{(\H_D^1)^*} \leq c \Vert g \Vert_{\BMO_D}$ for some generic constant $c$.

		\item For any $\ell \in (\H_D^1)^*$ there is a unique $h \in \BMO_D$ such that $\ell = \ell_h$ and $\Vert h \Vert_{\BMO_D} \leq c \Vert \ell \Vert_{(\H_D^1)^*}$  for some generic constant $c$.
	\end{enumerate}
\end{theorem}

\begin{proof}
	\textbf{Proof of (i).}
	With \eqref{eq: H1-BMO pairing} at our disposal, the argument in \cite[p.\,632]{Coifman-Weiss} applies \textit{verbatim}.

	\textbf{Proof of (ii).} Let $\ell \in (\H_D^1)^*$. If $D$ is empty, the argument is classical~\cite[p.\,632--634]{Coifman-Weiss}. Otherwise, we have to put in some extra effort to improve the classical estimate on balls near $D$. That being said, pick a reference ball $B'$ near $D$, which is possible in the current case. By definition, $B'$ also qualifies as a usual ball. Following the classical proof in the case $D = \varnothing$, there exists for any usual ball $B$ a mean value free function $F_B \in \L^2(B)$ such that $\| F_B \|_2 \leq |B|^{\nicefrac{1}{2}} \Vert \ell \Vert_{(\H_D^1)^*}$. We set $f_B \coloneqq F_B - (F_B)_{B'}$. As in the classical case, the function $f(x) \coloneqq f_B(x)$, whenever $x\in B$ is a usual ball, is well-defined. We claim that there is a constant $c$ depending on $\ell$ and $B'$ such that $h \coloneqq f + c \in \BMO_D$ and $\ell_h = \ell$.

	First, it follows readily that
	\begin{align}
		\left( \fint_{B} |f + (F_B)_{B'}|^2 \, \d x \right)^{\frac{1}{2}} = \left( \fint_{B} |F_{B}|^2 \, \d x \right)^{\frac{1}{2}} \leq \Vert \ell \Vert_{(\H_D^1)^*}.
	\end{align}
	Therefore, the $\BMO_D$-norm of $f$ is under control on usual balls. The same is true if we add a constant to $f$ later on.

	Second, to improve the bound on balls $\widetilde B$ near $D$ we construct an auxiliary function. For any such $\widetilde B$ there is a unique function $g_{\widetilde B} \in \L^2(\widetilde B)$ that represents $\ell$ restricted to $\L^2(\widetilde B)$ and fulfills the bound $\| g_{\widetilde B} \|_2 \leq |\widetilde B|^{\nicefrac{1}{2}} \Vert \ell \Vert_{(\H_D^1)^*}$.
	Now we can define a function $g$ by $g(x) = g_{\widetilde B}(x)$ whenever $x \in \widetilde B$ is a ball near $D$. The difference to the classical construction is that $g_{\widetilde B}$ is unique without the necessity to perform a normalization with respect to a reference ball.
	Testing $g-f$ with the mean value free functions of $\L^2(\widetilde B)$ shows that $g-f$ is constant on $\widetilde B$.
	Since $f$ is mean value free over the reference ball $B'$, averaging $g-f$ over $B'$ shows that this constant is $(g)_{B'}$ independent of the ball $\widetilde B$. Exhausting $O$ by such balls $\widetilde B$ yields $g = f + (g)_{B'}$. Therefore,
	\begin{align}
		\left( \fint_{\widetilde B} |f + (g)_{B'}|^2 \, \d x \right)^{\frac{1}{2}}
		\leq \left( \fint_{\widetilde B} |g|^2 \, \d x \right)^{\frac{1}{2}}
		\leq \Vert \ell \Vert_{(\H_D^1)^*}.
	\end{align}
	Thus, with $c = (g)_{B'}$ one has $h = f + c \in \BMO_D$ with $\Vert h \Vert_{\BMO_D} \leq c \Vert \ell \Vert_{(\H_D^1)^*}$ as claimed.
\end{proof}

	\subsection{Predual of the atomic Hardy space}
	\label{Subsection: Predual of H1}

	The main goal of this subsection is to determine a predual of $\H^1_D$. To this end, we introduce a $D$-adapted version of the space of vanishing mean oscillation in the spirit of Coifman and Weiss.

	\begin{definition}
		The \textbf{space of vanishing mean oscillation $\boldsymbol{\VMO_D}$} is the closure of %
		$\rC_{\cc}$
		with respect to $\Vert \cdot \Vert_{\BMO_D}$.
	\end{definition}

	Note that $\VMO_D$ is separable as $\rC_{\cc}$ contains even a $\Vert \cdot \Vert_{\infty}$-dense countable subset.

	\begin{remark}
		There is some flexibility in the definition of $\VMO_D$. Indeed, it is possible to take the closure of a suitable space of regular functions that vanish only in $D$. We will not pursue this direction in this article.
	\end{remark}

	The following is the main result of this section.

	\begin{theorem} \label{Predual of H1: Theorem: VMO* = H1}
		Under $\mathrm{(UITC)}$ the dual space of $\VMO_D$ is $\H^1_D$ up to equivalent norms.
	\end{theorem}

	This result naturally extends \cite[Thm.~4.1]{Coifman-Weiss} and we will need it in Section~\ref{Subsection: Inclusion of Operator-adapted space into H_D^1(O)} to show the inclusion $\H_L^1 \sub \H_D^1$. For its proof, we adapt the arguments in \cite{Coifman-Weiss}.
	We are going to use a trick to circumvent the \textit{a priori} norming property of $\rC_{\cc}$ for $\H_D^1$ used in~\cite{Coifman-Weiss}, which we learned from~\cite[p.\,85]{Zorin-Kranich_Master-Thesis}.

	\textbf{The inclusion $\boldsymbol{\H_D^1 \sub (\VMO_D)^*}$.} This one is rather easy to see. Indeed, the proof of Lemma~\ref{Precise Definitions: Lemma: H1-BMO-pairing} reveals that this inclusion map is a contraction. Moreover, injectivity follows from the fact that $\rC_{\cc}$ separates the points of $\H_D^1 \sub \L^1$.

	\textbf{The inclusion $\boldsymbol{(\VMO_D)^* \sub \H_D^1}$.} By the closed graph theorem, it suffices to show the inclusion in the set theoretic sense. The proof consists of two parts. The first one is a purely functional analytic reduction argument and the second one is a particular $D$-adapted modification of the argument in \cite{Coifman-Weiss}.

Write $\ball_X$ for the closed unit ball in a Banach space $X$.

\textbf{Step 1.} Pick $\ell \in (\VMO_D)^*$ and assume without loss of generality that $\ell \neq 0$. After normalization we can assume that $\Vert \ell \Vert_{(\VMO_D)^*} \leq 1$. Let $\widetilde{\ell} \in (\BMO_D)^*$ be a Hahn--Banach extension of $\ell$. By Theorem~\ref{H1-BMO-duality: Theorem: Main result} there is a universal constant $c > 0$ such that $\widetilde{\ell} \in c \ball_{(\H_D^1)^{**}}$. Now, Goldstine's theorem \cite[Thm.~8.3.17]{Werner_FA} implies that $c \ball_{\H_D^1}$ is $w^*$-dense in $c \ball_{(\H_D^1)^{**}}$. Using this and the separability of $\VMO_D$, we find a sequence $(f_n)_n \sub c \ball_{\H_D^1}$ such that $f_n \to \ell$ in the $w^*$-topology of $(\VMO_D)^*$.

\textbf{Step 2.} The upshot of the previous discussion is that the proof of the inclusion $(\VMO_D)^* \sub \H_D^1$ boils down to proving that any bounded sequence in $\H_D^1$ admits a subsequence which converges to some $f \in \H_D^1$ in the $w^*$-topology of $(\VMO_D)^*$.

The first auxiliary result is in the spirit of \cite[Lem.~4.4]{Coifman-Weiss}.
When $D = \varnothing$, our proof gives a technically less involved version of their result in the Euclidean setting.
For non-trivial $D$ an additional argument is needed.

\begin{lemma}\label{Predual of H1: Lemma: H1-representation of f via dyadic grid structure}
    Suppose Assumption~\ref{ass:atomic}. Then any $f \in \H^1_D$ admits an atomic decomposition
    \begin{equation*}
    		f = \sum_{(k, i) \in \Z \times \Z^d} \lambda_{k,i} a_{k,i}  \quad \text{with} \quad \Vert f \Vert_{\H_D^1} \simeq \Vert (\lambda_{k,i})_{k,i} \Vert_{\ell^1}.
    	\end{equation*}
    	In addition, it holds:

    		\begin{enumerate}
    		\item  If $a_{k,i}$ is associated with $B_{k,i}$, then $r(\ball_{k,i}) \simeq 2^k$.

    		\item For fixed $k \in \N$, only finitely many balls of $\{ B_{k,i} \}_{i \in \Z^d}$ can intersect a compact set in $\overline{O}$.
    	\end{enumerate}
 \end{lemma}

 \begin{proof}
   	Take an atomic decomposition $f = \sum_j  \lambda_j a_j$ with $\lambda_j \in (0,\infty)$ satisfying $\sum_j \lambda_j \leq 2 \| f \|_{\H^1_D}$ and let $a_j$ be associated with $B_j \coloneqq B(x_j,r_j)$. For each $j \in \N$ there is a unique $k \in \Z$ such that $2^{k-1} < r_j \leq 2^k$. Now let $Q_{k,i}$ be the unique cube in the standard half open dyadic grid of cubes with sidelength $2^{k+1}$ that contains $x_j$. Express this relation by the notation $j \equiv (k,i) \in \Z \times \Z^d$.

	Put $J_{k,i} \coloneqq \{ j : j \equiv (k,i) \}$ and let $j \in J_{k,i}$. We denote by $\ball_{k,i}$ a ball centered in $Q_{k,i} \cap O$ of smallest radius such that $2Q_{k,i}$ is contained. By construction, its radius is comparable to $2^k$, so (i) is satisfied. Also, since for fixed $k \in \Z$ cubes from the dyadic grid $\{Q_{k,i} \}_{i \in \Z^d}$ are disjoint, a counting argument yields (ii).

   	To define a new atomic decomposition of $f$ put
   	\begin{align}
 		\lambda_{k,i} \coloneqq \sum_{j\in J_{k,i}} \lambda_j \quad \& \quad a_{k,i} \coloneqq \frac{\sum_{j\in J_{k,i}} \lambda_j a_j}{\lambda_{k,i}}.
	\end{align}
   	We check that $a_{k,i}$ is (up to a multiplicative constant) an $\H^1_D$-atom. First, since by construction $B_j \subseteq 2Q_{k,i} \subseteq \ball_{k,i}$ for $j\in J_{k,i}$, $a_{k,i}$ is supported in $\overline{B_{k,i}}$. Second, for $j\in J_{k,i}$ it follows from $r(\ball_j) \simeq 2^k \simeq r(\ball_{k,i})$ and (UITC) that $|B_j| \geq c^{-2} |B_{k,i}|$ for an absolute constant $c > 0$. Hence,
   	\begin{align}
 		\| a_{k,i} \|_2 \leq (\lambda_{k,i})^{-1} \sum_{j\in J_{k,i}} |\lambda_j| \| a_j \|_2 \leq (\lambda_{k,i})^{-1} \sum_{j\in J_{k,i}} |\lambda_j| |B_j|^{-\frac{1}{2}} \leq c |B_{k,i}|^{-\frac{1}{2}}.
   	\end{align}
   	Third, if $B_{k,i}$ is usual, then already $B_j$ was usual for all $j\in J_{k,i}$, so $a_{k,i}$ still has mean value zero.
   	In summary, $a_{k,i}$ is a constant multiple of an $\H^1_D$-atom.

	In summary, we deduce the atomic decomposition
   	\begin{align}
   		f = \sum_{k, i} (c \lambda_{k,i}) ( c^{-1} a_{k,i}) \quad \text{with} \quad \sum_{k, i} c \lambda_{k,i} \leq 2 c \| f\|_{\H^1_D}. & \qedhere
	\end{align}

 \end{proof}

The second key lemma is an extension of \cite[Lem.~4.2]{Coifman-Weiss}.

 \begin{lemma} \label{Predual of H1: Lemma: Bounded sequence in H1}
 	 Assume $\mathrm{(UITC)}$. Any bounded sequence $(f_n)_n \sub \H_D^1$ has a subsequence, which converges to some $f \in \H_D^1$ in the $w^*$-topology of $(\VMO_D)^*$. Moreover, it holds $\| f \|_{\H^1_D} \lesssim \sup_{n \to \infty} \| f_n \|_{\H^1_D}$.
 \end{lemma}

    \begin{proof}
    	 After normalization we can assume that $\sup_n \Vert f_n \Vert_{\H_D^1} \leq 1$. Lemma~\ref{Predual of H1: Lemma: H1-representation of f via dyadic grid structure} furnishes for all $n \in \N$ a decomposition $f_n = \sum_{k, i} \lambda_{k,i}(n) a_{k,i}(n)$ with $\lambda_{k,i}(n) \in (0,\infty)$ and $\sup_n \sum_{k,i} \lambda_{k,i}(n) \les 1$. After passing to an appropriate subsequence, the Bolzano--Weierstrass theorem joint with a diagonalization argument allows us to assume that $\lambda_{k,i} \coloneqq \lim_{n \to \infty} \lambda_{k,i} (n)$ exists for all $(k,i)$ and $\sum_{k, i} \lambda_{k,i} \les 1$, see also~\cite[Lem.~4.3]{Coifman-Weiss}. Since $(a_{k,i}(n))_n \sub \L^2(B_{k,i})$ is bounded, there is some $a_{k,i} \in \L^2(B_{k,i})$ such that $a_{k,i}(n) \to a_{k,i}$ weakly in $\L^2(B_{k,i})$ along a subsequence. A diagonalization argument ensures that we can choose the same subsequence for all $(k,i)$. The weak convergence implies that $a_{k,i}$ inherits the size and (potentially) moment conditions from the atoms $(a_{k,i}(n))_n$. Thus, each $a_{k,i}$ is an $\H_D^1$-atom associated with $B_{k,i}$. We define the candidate
     \begin{equation*}
     	f \coloneqq \sum_{k, i} \lambda_{k,i} a_{k,i} \in \H_D^1.
     \end{equation*}
     By construction, $\| f \|_{\H^1_D} \lesssim \sup_{n \to \infty} \| f_n \|_{\H^1_D}$.
     It remains to show convergence.
     By a $3 \varepsilon$-argument it suffices to show
     \begin{equation*}
     	\langle f_n \, | \, \phi \rangle \to \langle f \, | \, \phi \rangle \mathrlap{\qquad (\phi \in \rC_{\cc}).}
     \end{equation*}
     Let $K \in \N$ to be chosen. Split
     \begin{equation*}
     	\langle f_n \, | \, \phi \rangle = \left(\sum_{k < - K} + \sum_{|k| \leq K}  + \sum_{k > K} \right) \sum_i \lambda_{k,i} (n) \int_{B_{k,i}} a_{k,i}(n) \phi \, \d x \eqqcolon \mathrm{(I)} + \mathrm{(II)} + \mathrm{(III)}.
     \end{equation*}
    Compared to the argument in \cite{Coifman-Weiss}, the major change lies in the treatment of the small cubes (I) as we have to consider atoms near $D$. Let $\varepsilon > 0$. We claim that $K$ can be chosen independently of $(k,i)$ such that $$|(\mathrm{I})| + |(\mathrm{III})| \les \varepsilon.$$ The same argument will also entail a similar estimate for the decomposition of $f$.

    \textbf{Estimate for (I).} If $B_{k,i}$ is usual, then $a_{k,i}(n)$ is average free over $B_{k,i}$. Denote by $x_{k,i}$ the center of $B_{k,i}$.
    Since $\phi$ is uniformly continuous, there is some $K \in \N$ such that $|\phi(x) - \phi(y)| \leq \varepsilon$ for all $x,y \in 2B_{k,i}$, where $i \in \Z^d$ and $k < -K$. Hence, with $y = x_{k,i}$ and using also the size condition for atoms, we infer
    \begin{equation*}
     \int_{B_{k,i}} |a_{k,i}(n)| |\phi - \phi(x_{k,i})| \, \d x \les \varepsilon.
    \end{equation*}
    Now, let $B_{k,i}$ be near $D$. Since $\phi$ vanishes in $D$, we can apply uniform continuity with $y\in D \cap 2B_{k,i}$ to give $|\phi(x)| \leq \varepsilon$ for all $x \in B_{k,i}$, $i \in \Z^d$ and $k < - K$. Consequently,
    \begin{equation*}
    	 \int_{B_{k,i}} |a_{k,i}(n) | |\phi | \, \d x \les \varepsilon.
    \end{equation*}
    Combining the last two estimates, we obtain
    \begin{equation*}
    	 |\mathrm{(I)}| \les \varepsilon \sum_{k < - K} \sum_i \lambda_{k,i}(n) \les \varepsilon.
    \end{equation*}
    \textbf{Estimate for (III).} Using the size condition for atoms, (UITC) and Lemma~\ref{Predual of H1: Lemma: H1-representation of f via dyadic grid structure} (i), we get
    \begin{equation*}
    	\mathrm{(III)} \leq \Vert \phi \Vert_2 \sum_{k > K} \sum_i \lambda_{k,i}(n) |B_{k,i}|^{-\frac{1}{2}} \les |\ball_{K,i}|^{-\frac{1}{2}} \Vert \phi \Vert_2 \sum_{k, i} \lambda_{k,i}(n) \les 2^{- K \frac{d}{2}} \Vert \phi \Vert_2.
    \end{equation*}
The right-hand side can be made smaller than $\varepsilon$ provided that $K$ is large enough.

Next, by Lemma~\ref{Predual of H1: Lemma: H1-representation of f via dyadic grid structure} (ii) and the compact support of $\phi$,
\begin{equation*}
	\mathrm{(II)} \to \int_O \sum_{|k| \leq K, i} \lambda_{k,i} a_{k,i} \phi \, \d x \qquad (n \to \infty).
\end{equation*}
It follows the estimate $$\limsup_{n \to \infty} |\langle f - f_n \, | \, \phi \rangle | \les \varepsilon,$$ which completes the proof.
    \end{proof}

\section{Identification of operator-adapted and atomic Hardy spaces}  \label{Section: Identification of operator-adapted and atomic Hardy spaces}

We introduce the operator-adapted spaces $\H^1_L$ and state and proof our main result (Theorem~\ref{Main result and applications: Theorem: H_D^1=H_L^1}). The proof heavily relies on the duality theory for atomic spaces developed in the last section. Throughout, we assume $(\mathrm{G}(\mu))$, (UITC), (Fat) and (LU).
Also (D) is used, but only once. We will make its usage explicit.

\subsection{Relevant definitions} \label{Subsection: Relevant definitions}

\begin{definition}[cones and square functions]
	Let $x \in O$ and $h > 0$.
	For a measurable function $f \colon (0, \infty) \times O \to \C$ define its \textbf{local square function} by
	\begin{equation*}
		(S^{\loc} f)(x) \coloneqq \left( \int_0^\dO \fint_{O(x,t)} |f(t,y)|^2 \, \frac{\d t \, \d y}{t} \right)^{\frac{1}{2}} \mathrlap{\qquad (x \in O).}
	\end{equation*}
	For either $\j(z) = \sqrt{z} \e^{-\sqrt{z}}$ or $\j(z) = z \e^{-z}$ set
	\begin{equation*}
		S_{\j, L}^{\loc}(f)(x) \coloneqq S^{\loc} (\j(\t^2 L) f)(x)  \mathrlap{\qquad (f \in \L^1, x \in O).}
	\end{equation*}
	Analogously, $S$ and $S^{\glob}$ are defined by the same expression as $S^{\loc}$ but with $(0,\dO)$ replaced by $(0,\infty)$ or $(\dO, \infty)$ in the $t$-integral, respectively. These square functions lead to the $L$-adapted square functions $S_{\j, L}$ and $S_{\j, L}^{\glob}$.
\end{definition}

\begin{remark}
	\label{rem:sf_normalization}
	Due to (UITC), the local square function can equivalently be expressed via
	\begin{align}
		(S^{\loc} f)(x) = \left( \int_0^{\dO} \int_{O(x,t)} |f(t,y)|^2 \, \frac{\d t \, \d y}{t^{1+d}} \right)^{\frac{1}{2}}, \quad x\in O.
	\end{align}
	This is not possible for $S^{\glob}$. We will come back to this later in~\eqref{eq:strange_monotonicity}.
\end{remark}

\begin{definition}

	The \textbf{operator-adapted Hardy space $\boldsymbol{\H_L^1}$} consists of all $f \in \L_0^1$ for which $S_{\j, L}^{\loc}(f) \in \L^1$, where either $\j(z) = \sqrt{z} \e^{-\sqrt{z}}$ or $\j(z) = z \e^{-z}$. Endow $\H_L^1$ with the norm
	\begin{align*}
		\Vert f \Vert_{\H_L^1} \coloneqq
			 \Vert S_{\j, L}^{\loc}(f) \Vert_1 + \Vert f \Vert_1.
	\end{align*}
\end{definition}

\begin{remark}[dependence on $\j$]
	At this stage, the definition of $\H_L^1$ using either the function $\j(z) = \sqrt{z} \e^{-\sqrt{z}}$ or $\j(z) = z \e^{-z}$ is a slight abuse of notation. We will of course treat both of them in all proofs. As a consequence of Theorem~\ref{Main result and applications: Theorem: H_D^1=H_L^1} it will turn out  that the definition of $\H_L^1$ is indeed independent of the choice of $\j$ up to equivalent norms.
\end{remark}

\begin{remark}[relation to the kernel--range splitting]
	\label{rem:fc_on_H1L}
	Suppose that $f\in \H^1_L \cap \L^2$. Then in particular $f\in \L^2_0$. According to Proposition~\ref{prop:kernel} this gives $f\in \overline{\sR(L)}$. A practical consequence is that if $\psi \in \H^\infty$ on a suitable sector, then $\psi(L) f \in \overline{\sR(L)} = \L^2_0$. Indeed, the functional calculus respects the kernel--range splitting.
\end{remark}

Let us show that $\H_L^1$ can be equipped with an equivalent norm. In the sequel we are going to use this fact freely.

\begin{proposition}[equivalent norm]
	\label{prop:H1L_renorming}
	The space $\H_L^1$ can be equipped with the equivalent norm $$\| f \|_{\H_L^1} \simeq \| S_{\j, L}(f) \|_1.$$
\end{proposition}

\begin{proof}
	Let $f\in \H_L^1$.

	\textbf{(1) $\boldsymbol{O}$ is unbounded}. Observe that $S_{\j, L}^{\loc} = S_{\j, L}$ since $O$ is unbounded.

	To control the $\L^1$-norm, suppose first that additionally $f\in \L^2$. Then the theory of pre-adapted spaces that we are going to introduce in Section~\ref{sec:maximal} yields $\| f \|_1 \lesssim \| S_{\j,L}(f) \|_1$, see Lemma~\ref{lem:pre-adapted_embed_L1}.

	Now, we treat a general $f$ in the Poisson case $\j(z) = \sqrt{z} \e^{-\sqrt{z}}$, with obvious modifications for $\j(z) = z \e^{-z}$. Owing to $\mathrm{G}(\mu)$ one has the convergence $\e^{-s \sqrt{L}} f \to f$ in $\L^1$ as $s\to 0$. For this approximation, one has the uniform bound $\sup_{s > 0} \| \e^{-s \sqrt{L}} f \|_{\H^1_L} \lesssim \| f \|_{\H^1_L}$ which we are going to present in full detail in Step~2 of the proof of the inclusion $\H_L^1 \sub \H_D^1$ (Section~\ref{Subsection: Inclusion of Operator-adapted space into H_D^1(O)}). Therefore, since $\e^{-s \sqrt{L}} f \in \L^2$, it follows
	\begin{align}
		\| f \|_1 = \lim_{s\to 0} \| \e^{-s \sqrt{L}} f \|_1 \lesssim \liminf_{s\to 0} \| \e^{-s \sqrt{L}} f \|_{\H^1_L} \lesssim \| f \|_{\H^1_L}.
	\end{align}
	Observe that this calculation did not use that $O$ is unbounded.

	\textbf{(2) $\boldsymbol{O}$ is bounded}. Recall that $S_{\j,L}(f) \leq S_{\j,L}^{\loc}(f) + S_{\j,L}^\glob(f)$. The same argument as in Step~1 gives $\| f \|_1 \lesssim \| S_{\j,L}(f) \|_1$, hence it suffices to show $\| S_{\j, L}^{\mathrm{glob}} (f) \|_1 \lesssim \| f \|_1$.

	Indeed, using Hölder's inequality, the averaging trick, the substitution $t \mapsto t + \nicefrac{\d(O)}{2}$, the square function estimate in $\L^2$ and $\mathrm{G}(\mu)$, we obtain
	\begin{align}
		\Vert S_{\j, L}^{\mathrm{glob}} (f) \Vert_1 &\leq |O|^{\frac{1}{2}} \Vert S_{\j, L}^{\mathrm{glob}} (f) \Vert_2
		\\&\simeq |O|^{\frac{1}{2}} \left(  \int_{\d(O)}^{\infty} \int_O | \sqrt{L} \e^{- (t - \frac{\d(O)}{2}) \sqrt{L}} \e^{- \frac{\d(O)}{2} \sqrt{L}}f(y)|^2 \, \d y \, \d t \right)^{\frac{1}{2}}
		\\&=  |O|^{\frac{1}{2}} \left(  \int_{\frac{\d(O)}{2}}^{\infty} \int_O | \sqrt{L} \e^{- t \sqrt{L}} \e^{- \frac{\d(O)}{2} \sqrt{L}}f(y)|^2 \, \d y \, \d t \right)^{\frac{1}{2}}
		\\&\les |O|^{\frac{1}{2}} \Vert S_{\j, L}(\e^{- \frac{\d(O)}{2} \sqrt{L}} f) \Vert_2
		\\&\les |O|^{\frac{1}{2}} \Vert \e^{- \frac{\d(O)}{2} \sqrt{L}} f \Vert_2
		\\&\les \Vert f \Vert_1. \qedhere
	\end{align}
\end{proof}

\subsection{Precise result and strategy of the proof}   \label{Subsection: Outline of the strategy}

Finally, we are able to state our main result.

\begin{theorem}  \label{Main result and applications: Theorem: H_D^1=H_L^1}
	Let $O \subseteq \R^d$ open and $D \subseteq \partial O$ closed. Suppose the geometric conditions $(\mathrm{UITC})$, $(\mathrm{LU})$, (D) and $(\mathrm{Fat})$. Moreover, suppose that $L$ satisfies $(\mathrm{G}(\mu))$.
	Then $$\H_L^1(O) = \H_D^1(O).$$
\end{theorem}

To prove it, we show both inclusions separately. The proof for $\H_L^1 \sub \H_D^1$ builds on the ideas in \cite[Sec.~3.4]{Auscher-Russ}. For the converse inclusion we choose a different route: the aforementioned work takes a detour through a maximal function characterization, which relies on the Lipschitz property of the domain. Taking a direct approach is one of the reasons why we can drastically lower the geometric requirements in our main result. In the self-adjoint case, we obtain nevertheless a maximal function characterization in Section~\ref{sec:maximal}. The strategy for the individual inclusions is as follows.

\textbf{$\boldsymbol{\H_D^1 \sub \H_L^1}$.} Taking Remark~\ref{rem:atomic_convergence} into account, this inclusion can be reinterpreted as proving boundedness of
\begin{equation*}
	S^{\loc}_{\j, L} \colon \H_D^1 \to \L^1.
\end{equation*}
Using the atomic decomposition for $\H_D^1$, it turns out that this follows from the uniform control on $\H_D^1$-atoms
\begin{equation*}
	\Vert S_{\j, L}^{\loc} (a) \Vert_{1} \les 1.
\end{equation*}
Let $a$ be associated with $B = \ball \cap O'$. To prove this bound, we split the $\L^1$-integral into a local and a global part, $4 B$ and $O' \setminus 4B$. The \enquote{local} part can be purely treated by the $\L^2$-theory of the operator. To control the \enquote{global} part, we split the $t$-integral in $(S_{\j, L}^{\loc} (a))(x)$ into several parts that let us exploit decay properties of the kernel stemming from $\mathrm{G}(\mu)$ combined with either the cancellation condition of $a$ or the vanishing of the kernel of $\j(t^2 L)$ in $D$ (obtained in Lemma~\ref{The geometric setup: Lemma: Kernel vanishes in D}), depending on whether the atom is usual or near $D$.

\textbf{$\boldsymbol{\H_L^1 \sub \H_D^1}$.}
We explain the case of unbounded $O$ first.
The main tool to show this inclusion is Theorem~\ref{Predual of H1: Theorem: VMO* = H1}. To make it applicable, we reproduce $f$ in the functional calculus for $L$.
After combining the reproducing formula with a duality estimate in tent spaces, it remains to prove that the (local) Carleson functional
\begin{equation*}
	C^{\loc}(\j(\t^2 L^*) \, \cdot \, ) \colon \VMO_D \to \L^{\infty}
\end{equation*}
that we are going to introduce in~\eqref{eq:def_carleson} is bounded. To this aim, we fix a ball $\ball$ centered in $O$ and decompose the test function $\phi \in \rC_{\cc}$ as $\phi = (\phi)_{B} + (\phi - (\phi)_{B})$. Split $\phi - (\phi)_{B}$ further into a \enquote{local} part $(\phi - (\phi)_{B}) \1_{2 B}$ and a \enquote{global} part $(\phi - (\phi)_{B}) \1_{O \setminus 2\ball}$. Similarly as in the proof of $\H_D^1 \sub \H_L^1$, the local part $C^{\loc}(\j(\t^2 L^*)(\phi - (\phi)_{B})\1_{2 B})$ can be controlled by the $\L^2$-theory of $L$ and the global part $C^{\loc}(\j(\t^2 L^*) (\phi - (\phi)_{B}) \1_{O \setminus 2 \ball})$ using decay properties of the kernel $q_t$. The main innovation lies in controlling $C^{\loc}(\j(\t^2 L^*)(\phi)_{B})$. There are three ingredients. %
First, we use a $D$-adapted conservation property for $(\j(t^2 L))_{t > 0}$ (see Lemma~\ref{Inclusion of Operator-adapted space into H_D^1(O): Lemma: D-adapted conservation property}). Second, uniform porosity of $D$ is used to estimate the size of boundary strips close to $D$. Let us mention that this is the first and last time in the present work that we need this information. Third, our definition of $\H^1_D$-atoms lets us treat certain large balls as boundary terms of a telescopic sum. At this point, we should mention that this change of definition also leads to a simplified argument in the case of pure Dirichlet boundary conditions compared to~\cite[Lem.~15]{Auscher-Russ}.

If $O$ is bounded, more technical challenges appear, notably a reduction to small scales in the reproducing formula and a localization of $B$ to multiple components. In each component, a different constant is subtracted in the decomposition of $\phi$.

The rest of this section is devoted to filling in the details.

\subsection{The inclusion \boldmath$\H_D^1 \sub \H^1_L $} \label{Section: Inclusion of H_D^1(O) into Operator-adapted space}

For this part we need (UITC), (Fat), (LU) and $(\mathrm{G}(\mu))$. The following lemma is central. Recall the kernel $q_t$ from Definition~\ref{def:qt}.

\begin{lemma} \label{Inclusion of H_D^1(O) into Operator-adapted space: Lemma: Interpolation bound}
	Assume $\mathrm{(Fat)}$, $\mathrm{(LU)}$ and $(\mathrm{G}(\mu))$. Let $\theta \in [0,1]$ and $a$ be an $\H_D^1$-atom associated with $B = B(x,r)$. Then it holds
	\begin{equation*}
		\left| \int_{O(x,r)} q_t (z,y) a(z) \, \d z \right| \les t^{-d} \left( \frac{r}{t} \right)^{\theta \mu} \left( 1 + \frac{\d_B(y)}{t} \right)^{-(d+1)(1 - \theta)}
	\end{equation*}
	for all $y \in O$ and $0 < t < \dO$.
\end{lemma}

\begin{proof}
	Fix $y \in O$ and $0 < t < \dO$. By \eqref{eq: Kernel bound for Poisson semigroup} and the size condition on $a$ we obtain
	\begin{align}
		\label{eq:kernel_atom1}
		\begin{split}
			\left| \int_{O(x,r)} q_t(z,y) a(z) \, \d z \right| &\les t^{-d} \left( 1 + \frac{\d_B(y)}{t} \right)^{-(d+1)} \Vert a \Vert_1 \\
			&\leq t^{-d} \left( 1 + \frac{\d_B(y)}{t} \right)^{-(d+1)}.
		\end{split}
	\end{align}
	Next, if $B$ is away from $D$, then $a$ satisfies the moment condition. Hence, the version of~\eqref{eq: Hoelder Kernel bound for Poisson semigroup} for $q_t$ delivers
	\begin{align}
		\label{eq:kernel_atom2}
		\begin{split}
			\left| \int_{O(x,r)} q_t(z,y) a(z) \, \d z \right| &= \left| \int_{O(x,r)} (q_t(z,y) - q_t(x, y) ) a(z) \, \d z \right| \\
		&\les t^{-d} \left( \frac{r}{t} \right)^{\mu} \Vert a \Vert_1 \\
		&\leq  t^{-d} \left( \frac{r}{t} \right)^{\mu}.
		\end{split}
	\end{align}
	If $B$ is near $D$, let $x_B \in 2 \ball \cap D$. Lemma~\ref{The geometric setup: Lemma: Kernel vanishes in D} gives $q_t(x_B, y) = 0$. Then a similar calculation as before in which $q_t(x_B, y)$ takes the role of $q_t(x, y)$ gives~\eqref{eq:kernel_atom2} in this situation.

	We conclude by taking logarithmic convex combinations between~\eqref{eq:kernel_atom1} and~\eqref{eq:kernel_atom2}.
\end{proof}

\begin{proof}[\rm\bf{Proof of the inclusion $\boldsymbol{\H_D^1 \sub \H_L^1}$.}] Since $\H_D^1 \sub \L^1$ is a contraction it suffices to prove that
\begin{equation*}
	\Vert S_{\j, L}^{\loc} (f) \Vert_1 \les \Vert f \Vert_{\H_D^1}.
\end{equation*}
We tacitly use Remark~\ref{rem:sf_normalization} in this proof.
Let us assume for the moment that there is some $C > 0$ such that
	\begin{equation}
		\Vert S_{\j, L}^{\loc}(a) \Vert_1 \leq C \label{eq: Inclusion of H_D^1(O) into Operator-adapted space: Theorem: Main result, estimate for atom}
	\end{equation}
	holds true for any $\H_D^1$-atom $a$. Let $f \in \H_D^1$ and fix some atomic decomposition $f = \sum_j \lambda_j a_j$ with $\Vert (\lambda_j)_j \Vert_{\ell^1} \leq 2 \Vert f \Vert_{\H_D^1}$. As $f = \sum_j \lambda_j a_j$ converges in $\L^1$, property $\mathrm{G}(\mu)$ implies for all $0 < t < \dO$ fixed that
	\begin{equation*}
		\j(t^2 L) f = \sum_j \lambda_j \j(t^2 L) a_j
	\end{equation*}
	converges in $\L^2$. This yields for every $x \in O$ the estimate
	\begin{equation*}
		\Vert \j(t^2 L) f \Vert_{\L^2(O(x,t))}^2 \leq \Bigl( \sum_j |\lambda_j| \Vert \j(t^2 L) a_j \Vert_{\L^2(O(x,t))} \Bigr)^2,
	\end{equation*}
	which (together with monotone convergence) entails that
	\begin{align*}
		S_{\j, L}^{\loc} (f)(x) &= \left( \int_0^{\dO} \Vert \j(t^2 L) f \Vert_{\L^2(O(x,t))}^2 \, \frac{\d t}{t^{1+d}} \right)^{\frac{1}{2}}
		\\&\leq \lim_{n \to \infty} \left( \int_0^{\dO} \left( \sum_{j=1}^n |\lambda_j| \Vert \j(t^2 L) a_j \Vert_{\L^2(O(x,t))} \right)^2 \, \frac{\d t}{t^{1+d}} \right)^{\frac{1}{2}}
		\\&\leq \sum_j |\lambda_j| S_{\j, L}^{\loc}(a_j)(x).
	\end{align*}
	Taking $\L^1$-norms on both sides, applying monotone convergence and using \eqref{eq: Inclusion of H_D^1(O) into Operator-adapted space: Theorem: Main result, estimate for atom} gives
	\begin{equation*}
		\Vert S_{\j, L}^{\loc}(f) \Vert_1 \leq C \sum_j |\lambda_j| \leq 2 C \Vert f \Vert_{\H_D^1}.
	\end{equation*}
	\textbf{Proof of \eqref{eq: Inclusion of H_D^1(O) into Operator-adapted space: Theorem: Main result, estimate for atom}.} Let $a$ be an $\H_D^1$-atom associated with $B = \ball(x_0, r) \cap O'$. The proof is divided into two parts.

	\textbf{Local part.} By Hölder's inequality and the $\L^2$-boundedness of $S_{\j, L}$, we get
	\begin{equation*}
		\Vert S_{\j, L}^{\loc}(a) \Vert_{\L^1(4\ball \cap O)} \leq |4\ball \cap O|^{\frac{1}{2}} \Vert S_{\j, L}^{\loc} (a) \Vert_2 \les |4\ball \cap O|^{\frac{1}{2}} \Vert a \Vert_2.
	\end{equation*}
	Now, (UITC) for $O$ and all its components along with the size condition of $a$ provide the uniform bound
	\begin{equation*}
		\Vert S_{\j, L}^{\loc}(a) \Vert_{\L^1(4\ball \cap O)} \les 1.
	\end{equation*}
	\textbf{Global part.} Fix $x \in O \setminus 4\ball$ and note that $r < \nicefrac{\d_B(x)}{2}$. To find upper bounds for
	\begin{equation*}
		S_{\j, L}^{\loc}(a)(x) = \left( \int_0^{\dO} \int_{O(x,t)} \left| \int_{B} q_t(z,y) a(z) \, \d z \right|^2 \, \frac{\d y \, \d t}{t^{1+d}} \right)^{\frac{1}{2}},
	\end{equation*}
	we split the time integral into the three parts
	\begin{equation*}
		0 < t \leq r, \quad r < t \leq \nicefrac{\d_B(x)}{2} \quad \& \quad \nicefrac{\d_B(x)}{2} \leq t < \dO.
	\end{equation*}
	We denote the corresponding integrals by (I), (II) and (III), respectively. Hence,
	\begin{equation*}
		S_{\j, L}^{\loc}(a)(x) \leq \mathrm{(I)} + \mathrm{(II)} + \mathrm{(III)}.
	\end{equation*}
	Before we start estimating the right-hand side, we observe that if $t \leq \nicefrac{\d_B(x)}{2}$ and $y \in O(x,t)$, then
	\begin{equation}
			\d_B(y) \geq \d_B(x) - |x-y| \geq \d_B(x) - t \geq \frac{\d_B(x)}{2}. \label{eq: H_D^1 sub H_L^1: Distance estimate for small times}
	\end{equation}
	\textbf{Estimate for (I).} Lemma~\ref{Inclusion of H_D^1(O) into Operator-adapted space: Lemma: Interpolation bound} with $\theta = 0$ joint with \eqref{eq: H_D^1 sub H_L^1: Distance estimate for small times} gives
	\begin{align*}
		\mathrm{(I)} \les \left( \int_0^{r} \int_{O(x,t)} \left( t^{-d} \d_B(x)^{-(d+1)} t^{d+1} \right)^2 \, \frac{\d y \, \d t}{t^{1+d}} \right)^{\frac{1}{2}} \les r \, \d_B(x)^{-(d+1)}.
	\end{align*}
	\textbf{Estimate for (II).} Fix $\theta \in (0,1)$ such that $d + \theta \mu > (d+1) (1 - \theta) > d$. We use this $\theta$ in Lemma~\ref{Inclusion of H_D^1(O) into Operator-adapted space: Lemma: Interpolation bound} along with \eqref{eq: H_D^1 sub H_L^1: Distance estimate for small times} to obtain
	\begin{align*}
		\mathrm{(II)} &\les \left( \int_{r}^{\frac{\d_B(x)}{2}} \int_{O(x,t)} \left( t^{-d} \left( \frac{r}{t} \right)^{\mu \theta} t^{(d+1)(1 - \theta)} \d_B(x)^{-(d+1)(1- \theta)} \right)^2 \, \frac{\d y \, \d t}{t^{1+d}} \right)^{\frac{1}{2}}
		\\& \les r^{\mu \theta} \d_B(x)^{-(d+1)(1-\theta)} \left( \int_{r}^{\frac{\d_B(x)}{2}} t^{-2( d + \mu \theta - (d+1)(1- \theta))} \, \frac{\d t}{t} \right)^{\frac{1}{2}}
		\\& \les  r^{(d+1)(1 - \theta) - d} \d_B(x)^{-(d+1)(1-\theta)}.
	\end{align*}
	\textbf{Estimate for (III).} Lemma~\ref{Inclusion of H_D^1(O) into Operator-adapted space: Lemma: Interpolation bound} with $\theta = 1$ gives
	\begin{align*}
		\mathrm{(III)} &\les \left( \int_{\frac{\d_B(x)}{2}}^{\infty} \int_{O(x,t)} \left( t^{-d} \left( \frac{r}{t} \right)^{\mu} \right)^2 \, \frac{\d y \, \d t}{t^{1+d}} \right)^{\frac{1}{2}}
		\\&\les  r^{\mu} \left( \int_{\frac{\d_B(x)}{2}}^{\infty} t^{-2 (d + \mu)} \, \frac{\d t}{t} \right)^{\frac{1}{2}}
		\\&\simeq r^{\mu} \d_B(x)^{-(d + \mu)}.
	\end{align*}
	Let $\gamma_1 \coloneqq 1$, $\gamma_2 \coloneqq (d+1)(1-\theta) -d$, $\gamma_3 \coloneqq \mu$ and recall that $\gamma_2 > 0$. Combining all three bounds and integrating over $O \setminus 4\ball$ yields
	\begin{align*}
		\Vert S_{\j, L}^{\loc}(a) \Vert_{\L^1(O \setminus 4\ball)} &\les \sum_{i=1}^3 r^{\gamma_i} \int_{O \setminus 4\ball} \d_B(x)^{-(d + \gamma_i)} \, \d x \les \sum_{i=1}^3 r^{\gamma_i} \int_{4 r}^{\infty} s^{-\gamma_i} \, \frac{\d s}{s}  \simeq 1.
	\end{align*}
	This completes the proof.
\end{proof}

	\begin{corollary}
		\label{cor:H1z_inclusion}
		If $O$ is unbounded and $D = \varnothing$, then one has the inclusion $\H^1_\zero \subseteq \H^1_L$.
	\end{corollary}

	\begin{proof}
		Using a trick from~\cite[Sec.~3.2]{Auscher-Russ}, the conclusion follows from the same calculation as presented above. The argument is as follows. The kernel $q_t$ on $O$ can be extended to a kernel $\widetilde q_t$ on $\R^d$ still satisfying~\eqref{eq: Kernel bound for Poisson semigroup} and~\eqref{eq: Hoelder Kernel bound for Poisson semigroup}. We need that $O$ is unbounded since otherwise the kernel bounds for $\widetilde q$ are only valid for finite times $t$ and not up to $\diam(\R^d) = \infty$. Now take $f\in \H^1_\zero$. Its zero extension $F$ from $O$ to $\R^d$ is in $\H^1(\R^d)$. As is classical, $F$ admits an atomic decomposition $F = \sum_j \lambda_j a_j$ into usual $\H^1_{\varnothing}(\R^d)$-atoms. Observe that
		\begin{align}
			\int_{O(x_0,r)} q_t(z,y) f(z) \, \d z = \int_{\ball(x_0,r)} \widetilde q_t(z,y) F(z) \, \d z
		\end{align}
		by the support property of $f$. Therefore, all arguments from before apply verbatim when the operator $\j(t^2L)$ is replaced by the operator on $\R^d$ given by the kernel $\widetilde q_t$.
	\end{proof}

	\begin{proof}[\rm\bf{Proof of Theorem~\ref{thm:H1z}.}]
		Admitting our abstract main result (Theorem~\ref{Main result and applications: Theorem: H_D^1=H_L^1}), Corollary~\ref{cor:H1z_inclusion} lets us conclude $\H^1_\zero \subseteq \H^1_L = \H^1_{\varnothing}$. The latter space coincides with $\H^1_\CW$, see also Proposition~\ref{prop:atomic_pure_cases}. We have already mentioned in Section~\ref{Section: Duality theory for D-adapted spaces} that the converse inclusion $\H^1_\CW \subseteq \H^1_\zero$ is immediate from the definitions, which completes the proof.
	\end{proof}

	\subsection{The inclusion \boldmath$\H^1_L \sub \H^1_D$}   \label{Subsection: Inclusion of Operator-adapted space into H_D^1(O)}

   In this part we explicitly make use of (D), (UITC) and $(\mathrm{G}(\mu))$, even though the other assumptions from Theorem~\ref{Main result and applications: Theorem: H_D^1=H_L^1} are implicitly used in some geometric reductions.
   We start with the $D$-adapted conservation property alluded in the strategy.

   \begin{lemma}[$D$-adapted conservation property] \label{Inclusion of Operator-adapted space into H_D^1(O): Lemma: D-adapted conservation property}
   	Let $O'$ be a component of $O$. It holds
   	\begin{equation}
   		\label{eq:D_conservation}
   		\left| \int_{O'} q_t(x,y) \, \d x \right| \les \left( 1 + \frac{\d_{D \cap \partial O'}(y)}{t} \right)^{-1} \qquad (0 < t < \dO, y \in O').
   	\end{equation}
   If $D \cap \partial O' = \varnothing$ one has $| \int_{O'} q_t(x,y) \, \d x | = 0$. With the convention $\d_{\varnothing}(y) = \infty$, this is consistent with the formula given above.
   \end{lemma}

   \begin{proof}
   	The case $D \cap \partial O' = \varnothing$ either follows from the classical conservation property in the case of pure Neumann boundary conditions or from the subsequent proof with obvious modifications. Hence, we assume that $\partial O'$ hits $D$.

   	Let $O'$ be such a component of $O$ and fix $y \in O'$. Let $r \coloneqq \d_{D \cap \partial O'}(y)$ and $B \coloneqq B(y,r)$.

   	\textbf{(1) Case} $\boldsymbol{\varphi(z) = z \mathrm{e}^{-z}}$.
   	We show in a first step that
   	\begin{equation}
   		\left| \int_{O'} t^2 \partial_t K_t(x,y) \, \d x \right| \les \e^{- \beta \left( \frac{r}{t} \right)^2}  \qquad (0 < t < \dO). \label{eq: D-adapted conservation property: Estimate for K}
   	\end{equation}
   	In particular, this implies~\eqref{eq:D_conservation} when $\varphi(z) = z \e^{-z}$.
   	Let $f \in \L^1$ with $\supp (f) \sub B(y, \nicefrac{r}{8})$ and $\Vert f \Vert_1 =1$. It suffices to prove that
   	\begin{align}
   		\label{eq: D-adapted conservation property: Estimate for K, splitting}
   		\Bigl| \int_{O'} t^2 L \e^{-t^2 L} f \,\d x \Bigr| \les \e^{- \beta \left( \frac{r}{t} \right)^2}.
   	\end{align}
   As usual, $\beta$ is allowed to change from line to line.
   Indeed, assuming \eqref{eq: D-adapted conservation property: Estimate for K, splitting} and taking Remark~\ref{Property G(mu): Remark: G(mu) stability} into account, we get
   \begin{align}
   		\label{eq:cons_exp_K}
   		\left| \int_{O'} \int_{B(y, \frac{r}{8})} t^2 \partial_t K_t(x,z) f(z) \, \d z \, \d x \right| \les \e^{- \beta \left( \frac{r}{t} \right)^2}.
   	\end{align}
   	Now, pick a non-negative $g \in \smooth[\ball(y,1)]$ with $\int_{\R^d} g \, \d x =1$, put $g_n(x) \coloneqq n^d g(nx)$ for all $n \in \N$ and choose $N \in \N$ such that $\nicefrac{1}{N} \leq \nicefrac{r}{8}$ and $\ball(y, \nicefrac{2}{N}) \sub O'$. For $n \geq N$ calculate
   	\begin{align}
   		\label{eq:cons_split}
   		&\left| \int_{O'} \int_{\ball(y, \frac{1}{n})} t^2 \partial_t K_t (x,z) g_n(z) \, \d z \, \d x - \int_{O'} t^2 \partial_t K_t (x,y) \, \d x \right|
   		\\&\leq \int_{O'} \int_{\ball(y, \frac{1}{n})} |t^2 \partial_t K_t (x,z) - t^2 \partial_t K_t(x,y)| g_n(z) \, \d z \, \d x.
   	\end{align}
	By construction,~\eqref{eq:cons_exp_K} can be applied with $f = g_n$ for all $n \geq N$. Therefore,~\eqref{eq: D-adapted conservation property: Estimate for K} follows if the right-hand side of~\eqref{eq:cons_split} goes to $0$ as $n\to \infty$.
   	To this end, split its $x$-integral into the two parts $\ball(y, \nicefrac{2}{n})$ and $O' \setminus \ball(y, \nicefrac{2}{n})$. We denote the corresponding integrals by $\mathrm{H}_{\loc}$ and $\mathrm{H}_{\mathrm{glob}}$, respectively.

   	\textbf{Estimate for $\boldsymbol{\mathrm{H}_{\loc}}$.} We use the Hölder estimate (G3) for $t \partial_t K_t$ to get
   	\begin{align*}
   		\mathrm{H}_{\loc} \les t^{- d - \mu} \int_{\ball(y, \frac{2}{n})} \int_{\ball(y, \frac{1}{n})} |y-z|^{\mu} g_n(z) \, \d z \, \d x \les t^{- d - \mu} n^{- \mu -2d} \longrightarrow 0  \qquad (n \to \infty).
   	\end{align*}
   	\textbf{Estimate for $\boldsymbol{\mathrm{H}_{\mathrm{glob}}}$.} Let $\nu \in (0, \mu)$ and $x \in O' \setminus \ball(y, \nicefrac{2}{n})$. Consequently, $|y-z| \leq \nicefrac{1}{n} \leq \nicefrac{|x-y|}{2}$. Remark~\ref{Property G(mu): Remark: G(mu) stability} (ii) delivers the estimate
   	\begin{align*}
   		|t^2 \partial_t K_t(x,y) - t^2 \partial_t K_t (x, z)| &\les t^{-d - \nu} |y-z|^{\nu} \e^{- \beta \left( \frac{|x-y|}{t} \right)^2} \les_t \frac{1}{n^{\nu}} \e^{- \beta \left( \frac{|x-y|}{t} \right)^2}.
   	\end{align*}
   	Hence, we obtain that
   	\begin{align*}
   		\mathrm{H}_{\mathrm{glob}} \les_t n^{-\nu-d} \int_{O' \setminus \ball(y, \nicefrac{2}{n})} \e^{- \beta \left( \frac{|x-y|}{t} \right)^2} \, \d x \les n^{-\nu-d} \int_{0}^{\infty} \e^{- \beta \left( \frac{r}{t} \right)^2} r^d \, \frac{\d r}{r} \longrightarrow 0  \quad (n \to \infty).
   	\end{align*}
   	Therefore, it only remains to show \eqref{eq: D-adapted conservation property: Estimate for K, splitting}. To this end, let $\psi \in \smooth[\R^d]$ such that $ \1_{\ball(y, \nicefrac{r}{4})} \leq \psi \leq \1_{\ball(y, \nicefrac{r}{2})}$ and $\Vert \nabla \psi \Vert_{\infty} \les r^{-1}$. Decompose
   	\begin{align}
   		\Bigl| \int_O t^2 L \e^{-t^2 L} f \,\d x \Bigr| &\leq |( t^2 L \e^{-t^2 L} f \, | \, \psi)_2| + |( t^2 L \e^{-t^2 L} f \, | \, (1 - \psi))_2| \\
   		&\eqqcolon \mathrm{(I)} + \mathrm{(II)}.
   	\end{align}

   	\textbf{Estimate for (I).} For brevity write $\L^2 = \L^2(O')$ in the following. We bound
   	\begin{align*}
   		|( t^2 L \e^{-t^2 L} f \, | \, \psi)_2| &= t^2 |( A \nabla \e^{-t^2 L} f \, | \, \nabla \psi)_2| \les t^2 \Vert \1_{O' \setminus \frac{1}{4} B} \nabla \e^{- t^2 L} f \Vert_2 \Vert \nabla \psi \Vert_2
   		\\&\les r^{\frac{d}{2} -1} t \Vert \1_{O' \setminus \frac{1}{4} B} t \nabla \e^{- t^2 L} f \Vert_2.
   	\end{align*}
   	A particular consequence of $(\mathrm{G}(\mu))$ are $\L^1 - \L^2$ off-diagonal estimates for $(\e^{-t^2 L})_{t > 0}$.
   	It is well-known~\cite[Chap.~4]{Block} that then also $(t \nabla \e^{- t^2 L})_{t > 0}$ satisfies $\L^1 - \L^2$ off-diagonal estimates.
   	Hence,
   	\begin{equation*}
   		|( t^2 L \e^{-t^2 L} f \, | \, \psi)_2| \les \left( \frac{r}{t} \right)^{\frac{d}{2} -1} \e^{- \beta \left( \frac{r}{t} \right)^2} \les \e^{- \beta \left( \frac{r}{t} \right)^2}.
   	\end{equation*}
   	\textbf{Estimate for (II).} Taking~\cite[Rem.~4.8]{Block} into account, $(\mathrm{G}(\mu))$ implies that $(t^2 L \e^{- t^2 L})_{t > 0}$ satisfies $\L^1$ off-diagonal estimates. It follows
   	\begin{equation*}
   		|( t^2 L \e^{-t^2 L} f \, | \, (1-\psi))_2| \leq \Vert \1_{O' \setminus \ball(y, \nicefrac{r}{4})} t^2 L \e^{- t^2 L} f \Vert_1 \les \e^{- \beta \left(\frac{r}{t} \right)^2}.
   	\end{equation*}

   	\textbf{(2) Case} $\boldsymbol{\varphi(z) = \sqrt{z} \mathrm{e}^{-\sqrt{z}}}$.
   	We differentiate the subordination formula \eqref{eq: Property G(mu): Subordination for kernel} to get
   	\begin{align*}
   		t \partial_t p_t(x,y) &= \frac{2}{\sqrt{\pi}} \int_0^{\infty} \tfrac{t}{2 \sqrt{u}} (\partial_t K)_{\frac{t}{2 \sqrt{u}}}(x,y) \e^{-u} \, \frac{\d u}{\sqrt{u}}
   		\\&= \frac{2}{\sqrt{\pi}} \int_0^{\nicefrac{t^2}{4\dO}} \tfrac{t}{2 \sqrt{u}} (\partial_t K)_{\frac{t}{2 \sqrt{u}}}(x,y) \e^{-u} \, \frac{\d u}{\sqrt{u}}
   		\\&\quad + \frac{2}{\sqrt{\pi}} \int_{\nicefrac{t^2}{4\dO}}^{\infty} \tfrac{t}{2 \sqrt{u}} (\partial_t K)_{\frac{t}{2 \sqrt{u}}}(x,y) \e^{-u} \, \frac{\d u}{\sqrt{u}}.
   	\end{align*}
   	Now, we use Fubini's theorem, \cite[Lem.~A.1]{Auscher-Russ} for the first integral and \eqref{eq: D-adapted conservation property: Estimate for K} for the second one in order to deduce
   	\begin{align*}
   		\left| \int_O t \partial_t p_t(x,y) \, \d x \right| &\les \int_0^{\nicefrac{t^2}{4\dO}} \e^{-u} \, \frac{\d u}{\sqrt{u}} + \int_0^{\infty} \e^{- (1 + c \left( \frac{r}{t} \right)^2 )u} \, \frac{\d u}{\sqrt{u}}
   		\\&\les t \1_{[\dO < \infty]} + \left( 1 + \frac{r}{t} \right)^{-1}
   		\\&\les \left( 1 + \frac{r}{t} \right)^{-1},
   	\end{align*}
   	where we used in the last step that $O$ is bounded.
   \end{proof}

\begin{proof}[\rm\bf{Proof of the inclusion $\boldsymbol{\H_L^1 \sub \H_D^1}$.}]
	The proof is divided into two parts.

	\textbf{Step 1. $\boldsymbol{f \in \H^1_L \cap \overline{\sR(L)}}$.} Let $0<b<\infty$. We start with the reproducing formula
	\begin{align}
		f = (2 b \sqrt{L} +1) \e^{- 2 b \sqrt{L}} f + 4 \int_0^{b} \j(t^2 L) \j(t^2 L) f \, \frac{\d t}{t} \eqqcolon f_1 + f_2,
	\end{align}
	which can be obtained from an integration by parts. For $b \leq \dO$ and $\phi \in \rC_{\cc}$ we claim
	\begin{align}
		\label{eq:f_tent_estimate}
		\left| \int_0^b ( \j(t^2 L) f \, | \, \j(t^2 L^*) \phi )_2 \, \frac{\d t}{t} \right| \lesssim \| f \|_{\H^1_L} \| \phi \|_{\BMO_D}.
	\end{align}
	Then we can conclude as follows.
	If $O$ is bounded, we use the reproducing formula with $b \coloneqq \dO$. Owing to~\eqref{eq:f_tent_estimate}, Theorem~\ref{Predual of H1: Theorem: VMO* = H1} yields an atomic decomposition of $f_2$.
	Next, we show that $f_1$ is a finite linear combination of $\H^1_D$-atoms. Write $f_1 = \sum_{O' \in \Sigma} \1_{O'} f_1$.
	First, we deduce with $(\mathrm{G}(\mu))$ the bound
	\begin{align}
		\label{eq:f_1_size}
		\| (2 \dO \sqrt{L} +1) \e^{- 2 \dO \sqrt{L}} f \|_2 \lesssim \dO^{-\frac{d}{2}} \| f \|_1 \lesssim \dO^{-\frac{d}{2}} \| f \|_{\H^1_L}.
	\end{align}
	Note that $O' = B(x,\dO)$ for any $x\in O'$.
	Keep in mind that all components satisfy (ITC) by Lemma~\ref{lem:ITC_for_components}.
	Consequently, $\dO^d \simeq |\ball(x,\dO) \cap O'| = |B(x,\dO)|$.
	It follows that $\1_{O'} f_1$ satisfies (up to normalization) the localization and size conditions of an $\H^1_D$-atom associated with the ball $B(x,\dO)$. Moreover, $(2 \dO \sqrt{L} +1) \e^{- 2 \dO \sqrt{L}}$ preserves $\overline{\sR(L)} = \L^2_0$ by Remark~\ref{rem:fc_on_H1L}. Therefore, in the case $O' = O_m$, the cancellation condition is fulfilled as well.

	Next, consider the case that $O$ is unbounded. Let $\phi \in \rC_{\cc}$ and consider $|(f \, | \, \phi)_2|$. Arguing as in~\eqref{eq:f_1_size}, we deduce $|(f_1 \, | \, \phi)_2| \lesssim b^{-\frac{d}{2}} \| f \|_{\H^1_L} \| \phi \|_2$, which vanishes as $b \to \infty$. Since~\eqref{eq:f_tent_estimate} holds for all $0<b<\infty$, we infer $|(f \, | \, \phi)_2| \lesssim \| f \|_{\H^1_L} \| \phi \|_{\BMO_D}$.
	This allows to conclude using Theorem~\ref{Predual of H1: Theorem: VMO* = H1} right away.

	It only remains to prove~\eqref{eq:f_tent_estimate}. To do so, we rely on a duality estimate in tent spaces. This duality estimate involves the so-called \textbf{local Carleson functional} defined for a measurable function $g \colon (0, \dO) \times O \to \C$ by
	\begin{equation}
		\label{eq:def_carleson}
		C^{\loc} (g)(x) \coloneqq \sup_{\ball \ni x} \left( \int_0^{r(\ball)} \fint_{\ball \cap O} |g(t,y)|^2 \, \d y \, \frac{\d t}{t} \right)^{\frac{1}{2}} \qquad (x \in O).
	\end{equation}
	Here, the supremum is taken over all Euclidean balls centered in $O$ that contain $x$ and are of radius $r(\ball) < \dO$.
	The duality estimate reads then as follows.
	\begin{lemma}[{\cite[Lem.\@ 3.16 \& Cor.\@ 3.17]{Amenta-Tent-Spaces}}] \label{Inclusion of Operator-adapted space into H_D^1(O): Lemma: T1-Tinfty-Pairing}
		Assume $\mathrm{(UITC)}$. Let $S^{\loc}(f) \in \L^1$ and $C^{\loc} (g) \in \mathrm{L}^{\infty}$. Then we have for all $0<b\leq \dO$ the estimate
		\begin{equation*}
			\left| \int_0^b \int_O f(t,x) g(t,x) \, \d x \, \frac{\d t}{t} \right| \les \Vert S^{\loc} (f) \Vert_1 \Vert C^{\loc} (g) \Vert_{\infty}.
		\end{equation*}
	\end{lemma}

	Now, let $0< b \leq \dO$ and $\phi \in \rC_{\cc}$. For convenience, we assume that $\phi$ is normalized in $\BMO_D$. We invoke Lemma~\ref{Inclusion of Operator-adapted space into H_D^1(O): Lemma: T1-Tinfty-Pairing} to infer
		\begin{equation*}
			\left| \int_0^b ( \j(t^2 L) f \, | \, \j(t^2 L^*) \phi )_2 \, \frac{\d t}{t} \right| \les \Vert S_{\j, L}^{\loc} (f) \Vert_1 \Vert C^{\loc}( \j(\t^2L^*) \phi) \Vert_{\infty}.
		\end{equation*}
		To bound $\Vert C^{\loc}( \j(\t^2L^*) \phi) \Vert_{\infty}$, let $x \in O$ and $\ball$ be a ball centered in $x_\ball \in O$ that contains $x$ with $r = r(\ball) < \dO$.
		Recall that either $O$ consists of only one component or $O$ is bounded and all components satisfy (ITC) as well.
		Let $x_\ball \in O''$ where $O'' \in \Sigma$. It follows $$\ball \cap O \subseteq B \cup \Bigl( \bigcup_{z \in O' \cap B} B(z, 2r) \Bigr).$$
		To be more precise, the union is taken over all $O' \in \Sigma$ distinct to $O''$ such that $O' \cap B \neq \varnothing$. Whenever we are in this situation, we pick one fixed $z$ from this non-empty intersection.
		Write $$\phi =  \sum_{O' \in \Sigma} \phi \1_{O'} \eqqcolon \sum_{O' \in \Sigma} \phi_{O'}.$$
		Fix one component $O'$ and put either $\widetilde B = B$ if $O' = O''$ or $\widetilde B = B(z, 2r)$, where $z \in O'$ from the above decomposition. We split
		\begin{align}
			\label{eq:decomp_phi_O'}
			\phi_{O'} = (\phi)_{\widetilde B} \1_{O'} + (\phi - (\phi)_{\widetilde B}) \1_{2\widetilde B} + (\phi - (\phi)_{\widetilde B}) \1_{O' \setminus 2\widetilde \ball}.
		\end{align}
		The functions on the right-hand side belong to $\L^{\infty}$, but $\j(t^2 L^*)$ extends its action linearly to $\L^{\infty}$ for each $0 < t \leq r(\ball)$ according to Lemma~\ref{Preliminary results: Lemma: Linfty-extension of the semigroup} and taking $r(\ball) < \dO$ into account.
		We concentrate first on the case $O' \neq O_m$ and explain the necessary changes for $O' = O_m$ afterwards. Write $\phi_i \coloneqq \phi_{i,O'}$ for the three terms on the right-hand side of~\eqref{eq:decomp_phi_O'}.
		That being said, it suffices to bound
		\begin{equation*}
			\sum_{i=1}^3 \int_0^{r} \fint_{\ball \cap O} |(\j(t^2 L^*) \phi_i)(y)|^2 \, \frac{\d y \, \d t}{t} \eqqcolon \mathrm{(I)} + \mathrm{(II)}_{\loc} + \mathrm{(II)}_{\mathrm{glob}} \les 1.
		\end{equation*}
		\textbf{Estimate for $\boldsymbol{\mathrm{(II)}_{\loc}}$.} McIntosh's theorem and (UITC) yield
		\begin{equation*}
			\mathrm{(II)}_{\loc} \les |\ball \cap O|^{-1} \Vert \phi_2 \Vert_2^2 \simeq |2\widetilde B|^{-1} \Vert \phi_2 \Vert_2^2 = \fint_{2 \widetilde B} |\phi - (\phi)_{\widetilde B}|^2 \, \d x \les 1.
		\end{equation*}
			\textbf{Estimate for $\boldsymbol{\mathrm{(II)}_{\mathrm{glob}}}$.} Assumption $(\mathrm{G}(\mu))$ implies that
		\begin{align*}
			\mathrm{(II)}_{\mathrm{glob}} &\leq \int_0^{r} \fint_{\ball \cap O} \left( \int_{O' \setminus 2 \widetilde \ball} | q_t (z,y)| |\phi(z) - (\phi)_{\widetilde B}| \, \d z \right)^{2} \, \d y \, \frac{\d t}{t}
			\\&\les \int_0^{r} \fint_{\ball \cap O} \left( \int_{O' \setminus 2 \widetilde \ball} t^{-d} \left( 1 + \frac{|z-y|}{t} \right)^{-(d+1)} |\phi(z) - (\phi)_{\widetilde B}| \, \d z \right)^{2} \, \d y \, \frac{\d t}{t}.
		\end{align*}
		Recall the notation $C_j(\ball)$ for annuli. We decompose $O' \setminus 2 \widetilde \ball = \bigcup_{j \geq 1} C_j(\widetilde \ball) \cap O'$. Since $|z-y| \simeq 2^j r$ for all $y \in \ball \cap O$ and $z \in C_j(\widetilde \ball) \cap O'$, we can further bound the latter by
		\begin{align}
			\label{eq: Adapted spaces contained in classical: Estimate for phi_3}
			\mathrm{(II)_{\mathrm{glob}}}&\les \int_0^{r} \Biggl( \sum_{j \geq 1} t^{-d} \left( 1 + 2^j \frac{r}{t} \right)^{-(d+1)} \int_{C_j(\widetilde \ball) \cap O'}  |\phi(z) - (\phi)_{\widetilde B}| \, \d z \Biggr)^{2} \, \frac{\d t}{t} \notag
			\\&\les \left( \int_0^{r} \left( \frac{t}{r} \right)^2 \, \frac{\d t}{t} \right) \Biggl( \sum_{j \geq 1}  2^{-j} \fint_{2^{j+1} \widetilde B} |\phi(z) - (\phi)_{\widetilde B}| \, \d z \Biggr)^{2} \notag\\&\simeq \Biggl( \sum_{j \geq 1}  2^{-j} \fint_{2^{j+1} \widetilde B} |\phi(z) - (\phi)_{\widetilde B}| \, \d z \Biggr)^{2}.
		\end{align}
		Next, using (UITC), we estimate
		\begin{align*}
			&\quad \fint_{2^{j+1} \widetilde B} |\phi(z) - (\phi)_{\widetilde B}| \, \d z \\
			&\leq  \fint_{2^{j+1} \widetilde B} |\phi(z) - (\phi)_{2^{j+1} \widetilde B}| \, \d z + \sum_{k=0}^{j} |(\phi)_{2^{k+1} \widetilde B} - (\phi)_{2^k \widetilde B}|
			\\&\les 1 + \sum_{k=0}^j \fint_{2^{k} \widetilde B} |\phi(z) - (\phi)_{2^{k+1} \widetilde B}| \, \d z
			\\&\les 1 + \sum_{k=0}^j \fint_{2^{k+1} \widetilde B} |\phi(z) - (\phi)_{2^{k+1} \widetilde B}| \, \d z \\
			&\les j.
		\end{align*}
		Inserting this back into \eqref{eq: Adapted spaces contained in classical: Estimate for phi_3} results in the upper bound
		\begin{align*}
			\mathrm{(II)}_{\mathrm{glob}} \les \Biggl( \sum_{j \geq 1}  j 2^{-j} \Biggr)^{2} \les 1.
		\end{align*}
		\textbf{Estimate for (I).} %
		In view of Lemma~\ref{Inclusion of Operator-adapted space into H_D^1(O): Lemma: D-adapted conservation property} applied with $O'$ we have
		\begin{align*}
			\mathrm{(I)} \les \left( \int_0^{r} \fint_{\ball \cap O} \left( 1 + \frac{\d_{D \cap \partial O'}(y)}{t} \right)^{-2} \, \d y \, \frac{\d t}{t} \right) |(\phi)_{\widetilde B}|^2.
		\end{align*}
		Write $x_{\widetilde{B}}$ for the center of $\widetilde{B}$. To bound the integral on the right, we distinguish two cases.

		\textbf{(1) Case $\boldsymbol{\d_{D \cap \partial O'}(x_{\widetilde{B}}) < 2r }$.} By assumption $\widetilde{B}$ is  near $D$. Let
		\begin{equation*}
			f(y) \coloneqq \left( 1 + \frac{\d_{D \cap \partial O'}(y)}{t} \right)^{-2}  \qquad (y \in \widetilde B),
		\end{equation*}
		and note that
		\begin{equation*}
			f(y) > \lambda \quad \iff \quad \d_{D \cap \partial O'}(y) < t (\lambda^{- \frac{1}{2}} -1) \eqqcolon s.
		\end{equation*}
		The layer cake formula joint with the substitution $\d s = - \frac{t}{2} \lambda^{- \frac{3}{2}} \, \d \lambda$ gives
		\begin{align*}
			&\int_0^{r} \int_{\ball \cap O} \left( 1 + \frac{\d_{D \cap \partial O'}(y)}{t} \right)^{-2} \, \d y \, \frac{\d t}{t} = \int_0^{r} \int_0^1 |\{ y \in \ball \cap O :  f(y) > \lambda \}| \, \d \lambda \, \frac{\d t}{t}
			\\&= 2 \int_0^{r} \int_0^{\infty} |\{ y \in \ball \cap O: \d_{D \cap \partial O'}(y) < s \}| \left( 1 + \frac{s}{t} \right)^{-3} \, \d s \, \frac{\d t}{t^2}.
		\end{align*}
		Now, Lemma~\ref{The geometric setup: Lemma: Assumption (D), upper estimate for measure} furnishes some $\eta \in (0,1)$ such that
		\begin{align*}
			|\{ y \in \ball \cap O: \d_{D \cap \partial O'}(y) < s \}| \les r^{d - \eta} s^{\eta}.
		\end{align*}
		Hence, using the substitution $s \mapsto t s$, the fact that $\eta < 2$ and (UITC) we get the upper estimate
		\begin{align*}
			r^{d - \eta} \int_0^{r} \left( \int_0^{\infty} s^{\eta} \left( 1 + \frac{s}{t} \right)^{-3} \, \d s \right) \, \frac{\d t}{t^2} \simeq r^{d- \eta} \int_0^{r} t^{\eta} \, \frac{\d t}{t} \simeq r^d \les |\ball \cap O|.
		\end{align*}
		In summary, we have $\mathrm{(I)} \les |(\phi)_{\widetilde B}|^2$.
		But since $\widetilde B$ is near $D$, $|(\phi)_{\widetilde B}| \leq 1$ by Hölder's inequality and the very definition of the $\BMO_D$-norm. Therefore $\mathrm{(I)} \les 1$ and we are done with this case. Observe that the localization of $\phi$ to the components of $O$ is used in a crucial way here, for the $\BMO_D$-norm only gives control on $\widetilde B = \widetilde \ball \cap O'$ and not on  the full $\ball \cap O$.

		\textbf{(2) Case $\boldsymbol{2r \leq \d_{D \cap \partial O'}(x_{\widetilde{B}})}$.} In the current case one has $\d_{D \cap \partial O'}(x_{\widetilde{B}}) \simeq \d_{D \cap \partial O'}(y)$ for all $y \in \widetilde B$. Hence
		\begin{align}
			\label{eq:const1}
			\int_0^{r} \fint_{\ball \cap O} \left( 1 + \frac{\d_{D \cap \partial O'}(y)}{t} \right)^{-2} \, \d y \, \frac{\d t}{t} \les \d_{D \cap \partial O'}(x_{\widetilde{B}})^{-2} \int_0^{r} t^2 \, \frac{\d t}{t} \simeq \left( \frac{r}{\d_{D \cap \partial O'}(x_{\widetilde{B}})} \right)^2.
		\end{align}
		Now we treat $|(\phi)_{\widetilde B}|$. To this end, let $k \geq 1$ be minimal such that $2^{k+1} \widetilde \ball \cap (D \cap \partial O') \neq \varnothing$. Such $k$ exists since $O' \neq O_m$ and it follows that $2^{k} \widetilde B$ is a ball near $D$. Note that $2^k r \simeq \d_{D \cap \partial O'}(x_{\widetilde B})$. Thus, we get by a telescope sum and (UITC) that
		\begin{align}
			\label{eq:const2}
			|(\phi)_{\widetilde B}| &\leq \sum_{j=0}^{k-1} |(\phi)_{2^j \widetilde B} - (\phi)_{2^{j+1} \widetilde B}| + |(\phi)_{2^k \widetilde B}|
			\les k
			\simeq \ln \left(\frac{\d_{D \cap \partial O'}(x)}{r} \right).
		\end{align}
		Now since $\sup_{\lambda \geq 1} \nicefrac{\ln(\lambda)}{\lambda}$ is finite, we derive $\mathrm{(I)} \les 1$ from the last two estimates as desired.

		Finally, if $O' = O_m$ the terms $\mathrm{(II)}_{\mathrm{loc}}$ and $\mathrm{(II)}_{\mathrm{glob}}$ are handled the same way. The treatment of $(\mathrm{I})$ is even easier: this term vanishes right away in virtue of the conservation property on $O_m$. This completes the proof of Step~1.

		\textbf{Step 2. $\boldsymbol{f \in \H^1_L}$.} We reduce the general case to Step~1 by and suitable approximation of $f$ via functions in $\H^1_L \cap \overline{\sR(L)}$. We concentrate on the case $\j(z) = \sqrt{z} \e^{-\sqrt{z}}$, the other case follows with the obvious modifications. That being said, we claim that
		\begin{align}
			P_s f \in \H^1_L \cap \overline{\sR(L)} \quad (s > 0) \quad \& \quad  \sup_{0<s<\dO} \Vert P_s f \Vert_{\H^1_L} \les \Vert f \Vert_{\H^1_L}.   \label{eq: Adapted spaces contained in classical: Step 2, first claim}
		\end{align}
		In order to prove the claim, we modify an argument from \cite[Prop.\@ 4.5]{Auscher-Stahlhut} to our needs.
		First of all, $\mathrm{G}(\mu)$ and Remark~\ref{rem:fc_on_H1L} yield $P_s f \in \overline{\sR(L)}$ and $\Vert P_s f \Vert_1 \les \Vert f \Vert_1$ for all $0 < s < \dO$. Hence, in the light of Proposition~\ref{prop:H1L_renorming}, it suffices to show $\Vert S_{\j, L} (P_s f) \Vert_1 \les \| f \|_{\H^1_L}$.
		Fix $N > d+1$ and define $\j_N(z) \coloneqq z^{\nicefrac{N}{2}} \e^{- \sqrt{z}}$.
		Since $P_s f \in \H^1_L \cap \overline{\sR(L)}$, a result from the theory of pre-adapted spaces~\cite[Prop.~4.4]{Amenta-Auscher} provides $$\Vert S_{\j, L} (P_s f) \Vert_1 \les \Vert S_{\j_N, L}(P_s f) \Vert_1.$$
		Now we are going to deal with the issue announced in Remark~\ref{rem:sf_normalization}.
		Instead of simply working with the power function $\t^{N-d-1}$, we have to consider the function
		\begin{align}
			\label{eq:strange_monotonicity}
			\Phi: t \mapsto \frac{t^N}{t|O(x,t)|}.
		\end{align}
		We claim that it is essentially increasing, that is to say, $s < t$ implies $\Phi(s) \leq C \Phi(t)$ for an absolute constant $C>0$. Indeed, by (UITC) we only have to deal with the situation $s < \dO \leq t$. Then estimate with (UITC) and monotonicity properties
		\begin{align}
			\frac{s^N}{s|O(x,s)|} \lesssim s^{N-1-d} \leq \dO^{N-1-d} = \frac{\dO^{N-1}}{|O(x,t)|} \leq \frac{t^N}{t|O(x,t)|},
		\end{align}
		which gives the claim.
		It follows from~\eqref{eq:strange_monotonicity} and the substitution $t' = t+s$ that
		\begin{equation*}
			\Vert S_{\j_N, L}(P_s f) \Vert_1 \leq \Vert S_{\j_N,L}(f) \Vert_1 \lesssim \Vert S_{\j,L}(f) \Vert_1,
		\end{equation*}
		where the last step follows again from~\cite[Prop.~4.4]{Amenta-Auscher}. Note that this time we have to split $P_t = P_{\nicefrac{t}{2}} P_{\nicefrac{t}{2}}$ and change the aperture of the cone to ensure the technical constraint that~\cite[Prop.~4.4]{Amenta-Auscher} only applies to $\H^1_L \cap \overline{\sR(L)}$.
		This completes the proof of~\eqref{eq: Adapted spaces contained in classical: Step 2, first claim}.

		Now, we complete Step~2. By~\eqref{eq: Adapted spaces contained in classical: Step 2, first claim} and Step~1, $P_s f \in \H^1_D$ with $\| P_s f \|_{\H^1_D} \lesssim \| P_s f \|_{\H^1_L} \lesssim \| f \|_{\H^1_L}$. Therefore, we can apply Lemma~\ref{Predual of H1: Lemma: H1-representation of f via dyadic grid structure} to the sequence $( P_{\nicefrac{1}{n}} f )_n$. This yields $g \in \H^1_D$ with $\| g \|_{\H^1_D} \lesssim \| f \|_{\H^1_L}$ such that $\langle P_{\nicefrac{1}{n}} f \, | \, \phi \rangle \to \langle g \, | \, \phi \rangle$ for every $\phi \in \rC_{\cc}$ along a subsequence. But owing to $\mathrm{G}(\mu)$, $(P_s)_{s \geq 0}$ is strongly continuous on $\L^1$. Therefore, also $\langle P_{\nicefrac{1}{n}} f \, | \, \phi \rangle \to \langle f \, | \, \phi \rangle$. Since $\rC_{\cc}$ separates the points of $\L^1$, $f = g$ follows and completes the proof.
	\end{proof}

	\section{Maximal characterization for hermitian coefficients}
	\label{sec:maximal}

	Throughout, assume that $L$ satisfies $\mathrm{G}(\mu)$ and is self-adjoint. The aim of this section is to enrich the result derived in Theorem~\ref{Main result and applications: Theorem: H_D^1=H_L^1} by a maximal function characterization. This extends the results given in~\cite[Sec.~4]{JFA-Hardy} on the pure Dirichlet and Neumann cases to mixed boundary conditions. For the negative Laplacian subject to mixed boundary conditions on a bounded domain, our very mild geometric assumptions let us recover the full picture of classical Hardy space theory in this way, see Theorem~\ref{thm:laplace}.

	\begin{definition}
		\label{def:max_fctn}
		For $f\in \L^2$ we define the \textbf{non-tangential maximal function} adapted to $L$ by
		\begin{align}
			f_L^*(x) \coloneqq \sup_{t > 0} \sup_{y\in O(x,t)} |\e^{-tL} f(y)| \mathrlap{\qquad (x\in O),}
		\end{align}
		and the \textbf{radial maximal function} adapted to $L$ by
		\begin{align}
			f_L^+(x) \coloneqq \sup_{t > 0} |\e^{-tL} f(x)| \mathrlap{\qquad (x\in O).}
		\end{align}
	\end{definition}

	\begin{remark}
		As a consequence of~$\mathrm{G}(\mu)$, pointwise evaluation of $\e^{-tL} f$ is well-defined.
	\end{remark}

	This enables us to define $L$-adapted Hardy spaces using a maximal function characterization.

	\begin{definition}
		\label{def:H1_max}
		The \textbf{non-tangential} and \textbf{radial maximal pre-Hardy spaces} $\mathbb{H}^1_{L,\mathrm{max}}$ and $\mathbb{H}^1_{L,\mathrm{rad}}$ are defined as
		\begin{align}
			\mathbb{H}^1_{L,\mathrm{max}} \coloneqq \bigl\{ f\in \L^2 \colon f_L^* \in \L^1 \bigr\}, \qquad
			\|f\|_{\mathbb{H}^1_{L,\mathrm{max}}} \coloneqq \| f_L^* \|_1,
		\end{align}
		and
		\begin{align}
			\mathbb{H}^1_{L,\mathrm{rad}} \coloneqq \bigl\{ f\in \L^2 \colon f_L^+ \in \L^1 \bigr\}, \qquad
			\|f\|_{\mathbb{H}^1_{L,\mathrm{rad}}} \coloneqq \| f_L^+ \|_1.
		\end{align}
		Then the \textbf{non-tangential} and \textbf{radial maximal Hardy spaces} $\H^1_{L,\mathrm{max}}$ and $\H^1_{L,\mathrm{rad}}$ are defined as the completions of $\mathbb{H}^1_{L,\mathrm{max}}$ and $\mathbb{H}^1_{L,\mathrm{rad}}$, respectively.
	\end{definition}

	The abstract main result of this section is the following.

	\begin{theorem}
		\label{thm:max}
		Assume that the operator $L$ is self-adjoint and satisfies~$\mathrm{G}(\mu)$ for some $\mu \in (0,1]$. Suppose that the geometric requirements~$(\mathrm{UITC})$, $(D)$, $(\mathrm{LU})$ and $(\mathrm{Fat})$ are fulfilled. Then one has the identity
		\begin{align}
			\H^1_{L,\mathrm{max}} = \H^1_{L,\mathrm{rad}} = \H^1_L \oplus \ell^1(\ISet) = \H^1_{D} \oplus \ell^1(\ISet),
		\end{align}
		where $\ell^1(\ISet)$ is the space of locally constant functions introduced in Definition~\ref{def:locally_constant}.
	\end{theorem}

	The framework developed by the second-named author together with Ciani and Egert~\cite{BCE-Gauss-for-MBC} lets us conclude a powerful application for the negative Laplacian on a bounded open set.

	\begin{theorem}
		\label{thm:laplace}
		Let $O \subseteq \R^d$ be open and bounded, and assume $(\mathrm{ITC})$, $(D)$, $\mathrm{(Fat)}$ and $(\mathrm{LU})$. If $-\Delta_D$ denotes the negative Laplacian subject to mixed boundary conditions, then we have the identities
		\begin{align}
			\{ f\in \L^1 \colon S_{-\Delta_D}(f) \in \L^1 \} = \H^1_D \oplus \ell^1(\ISet) = \H^1_{-\Delta_D,\mathrm{max}} = \H^1_{-\Delta_D,\mathrm{rad}},
		\end{align}
		where the space on the left-hand side is equipped with the norm $\| f \| \coloneqq \| f \|_1 + \| S_{-\Delta_D}(f) \|_1$.
	\end{theorem}

	\begin{remark}
		\begin{enumerate}
			\item We wish to emphasize that both the space $\H^1_D \oplus \ell^1(\ISet)$ and the assumptions of the theorem are purely geometrical and independent of any operator.
			\item The theorem is formulated for the negative Laplacian for convenience, but it is true for any elliptic operator with real and self-adjoint coefficient matrix.
		\end{enumerate}
	\end{remark}

	We summarize as corollaries the pure Dirichlet and Neumann cases. This takes Proposition~\ref{prop:atomic_pure_cases} into account.

	\begin{corollary}[pure Dirichlet case]
		Let $O \subseteq \R^d$ be open and bounded, and assume $(\mathrm{ITC})$, $(\mathrm{Fat})$, and that $\partial O$ is porous.
		If $-\Delta_0$ denotes the negative Laplacian subject to pure Dirichlet boundary conditions, then we have the identities
		\begin{align}
			\{ f\in \L^1 \colon S_{-\Delta_0}(f) \in \L^1 \} = \H^1_\Miyachi = \H^1_{-\Delta_0,\mathrm{max}} = \H^1_{-\Delta_0,\mathrm{rad}},
		\end{align}
		where the space on the left-hand side is equipped with the norm $\| f \| \coloneqq \| S_{-\Delta_0}(f) \|_1$.
	\end{corollary}

	\begin{corollary}[pure Neumann case]
		Let $O \subseteq \R^d$ be open and bounded, and assume $(\mathrm{ITC})$ and $(\mathrm{LU})$.
		If $-\Delta_N$ denotes the negative Laplacian subject to pure Neumann boundary conditions, then we have the identities
		\begin{align}
			\{ f\in \L^1 \colon S_{-\Delta_N}(f) \in \L^1 \} = \H^1_\CW = \H^1_{-\Delta_N,\mathrm{max}} = \H^1_{-\Delta_N,\mathrm{rad}},
		\end{align}
		where the space on the left-hand side is equipped with the norm $\| f \| \coloneqq \| f \|_1 + \| S_{-\Delta_N}(f) \|_1$.
	\end{corollary}

	\subsection{Bridging the theories using pre-adapted spaces}

	We introduce a pre-adapted version $\mathbb{H}^1_L$ of our space $\H^1_L$. Such spaces can be defined and studied only assuming that the operator $L$ has a decent $\L^2$ theory, %
	which is notably weaker than our requirement~$(\mathrm{G}(\mu))$. %

	For the rest of this section, we suppose at least that $L$ is self-adjoint and that the set $O$ satisfies (UITC). Then $O$ is a space of homogeneous type and the semigroup generated by $-L$ satisfies Davies--Gaffney estimates. Hence, our constellation falls into the scope of~\cite{Hardy-Gaffney}.

	\begin{definition}
		The \textbf{pre-adapted Hardy space $\boldsymbol{\mathbb{H}^1_L}$} is defined by
		\begin{align}
			\mathbb{H}^1_L \coloneqq \bigl \{ f\in \overline{\sR(L)} \colon S_L(f) \in \L^1 \bigr\}, \qquad \| f \|_{\mathbb{H}^1_L} \coloneqq \| S_L(f) \|_1.
		\end{align}
	\end{definition}

	\begin{remark}
		The square function $S_L(f)$ can either denote the heat or Poisson version. The respective spaces in~\cite{Hardy-Gaffney} are denoted by $\mathbb{H}^1_{L,S_h}$ and $\mathbb{H}^1_{L,S_P}$. It turns out that they coincide~\cite[Thm.~2.5]{Hardy-Gaffney}, so we do not bother to distinguish them and simply write $\mathbb{H}^1_L$ and $S_L$.
	\end{remark}

	\begin{remark}
		\label{rem:pre_adapted_vs_adapted}
		In virtue of Proposition~\ref{prop:H1L_renorming}, the spaces $\mathbb{H}^1_L$ and $\H^1_L$ are precisely the functions $f$ such that $S_L(f) \in \L^1$ equipped with the same norm, where in the first case $f$ is quantified over $\overline{\sR(L)} = \L^2_0$ and in the second case over $\L^1_0$.
	\end{remark}

	In a nutshell, the line of reasoning to obtain Theorem~\ref{thm:max} is as follows.

	\begin{enumerate}
		\item The space $\H^1_L$ is complete and its intersection with $\overline{\sR(L)}$ is a dense subspace. It follows that $\H^1_L$ is a completion of $\mathbb{H}^1_L$.

		\item It holds $\mathbb{H}^1_L \oplus \ell^1(\ISet) = \mathbb{H}^1_{L,\mathrm{max}} = \mathbb{H}^1_{L,\mathrm{rad}}$. Therefore, taking completions gives $$\H^1_L \oplus \ell^1(\ISet) = \H^1_{L,\mathrm{max}} = \H^1_{L,\mathrm{rad}}.$$
	\end{enumerate}

	The next proposition takes care of the first item.

	\begin{proposition}[$\H^1_L$ is a completion of $\mathbb{H}^1_L$]
		\label{prop:completion}
		Assume that $L$ satisfies $(\mathrm{G}(\mu))$ as well as the geometric assumptions $(\mathrm{UITC})$, $(D)$, $(\mathrm{Fat})$ and $(\mathrm{LU})$.
		Then the space $\H^1_L$ is complete and contains $\mathbb{H}^1_L$ as a dense subspace. Phrased differently, $\H^1_L$ is a completion of the pre-adapted space $\mathbb{H}^1_L$.
	\end{proposition}

	\begin{proof}
		Many functional analytic properties of $\H^1_L$ are not evident from its very definition. Instead, we use the atomic viewpoint developed in Theorem~\ref{Main result and applications: Theorem: H_D^1=H_L^1}.

		Indeed, $\H^1_D$ is the dual space of $\VMO_D$ by Theorem~\ref{Predual of H1: Theorem: VMO* = H1} and therefore complete. Also, finite linear combinations of $\H^1_D$-atoms are dense in $\H^1_D = \H^1_L$. Recall that we write $\H^1_{D, a}$ for the space of finite linear combinations of $\H^1_D$-atoms. The second claim in the proposition follows if we can show the chain of inclusions
		\begin{align}
			\label{eq:incl_completion}
			\H^1_{D,a} \subseteq \mathbb{H}^1_L \subseteq \H^1_L.
		\end{align}
		Indeed, keeping Remark~\ref{rem:pre_adapted_vs_adapted} in mind, the second inclusion in~\eqref{eq:incl_completion} follows readily from the fact that $\mathbb{H}^1_L \subseteq \L^1$. We will discuss this in a moment using the notions of atomic or molecular decompositions, see Lemma~\ref{lem:pre-adapted_embed_L1}. For the first inclusion of~\eqref{eq:incl_completion} note that atoms are by definition elements of $\L^2_0 = \overline{\sR(L)}$ (see Proposition~\ref{prop:kernel}) and satisfy $S_L(a) \in \L^1$ by Theorem~\ref{Main result and applications: Theorem: H_D^1=H_L^1}.
	\end{proof}

	\begin{remark}
		This observation could have been drawn already in~\cite{Auscher-Russ} but was not explicitly stated.
	\end{remark}

	In both~\cite{Hardy-Gaffney} and~\cite{JFA-Hardy} they derive an $L$-adapted atomic decomposition, in the first case coming from $\mathbb{H}^1_L$, in the second case coming from $\mathbb{H}^1_{L,\mathrm{max}}$. In~\cite{Hardy-Gaffney}, they give the following definition.
	\begin{definition}[{$L$-atomic pre-Hardy spaces according to~\cite[Def.~2.1 \& 2.2]{Hardy-Gaffney}}]
		\label{def:L_adapted_atoms}
		Let $M \in \N$. A function $a \in \L^2$ is called a $(1,2,M)$-atom associated to $L$ if there exists $b \in \D(L^M)$ and a ball $\ball$ of radius $r > 0$ centered in $O$ such that
		\begin{align}
			(\mathrm{i}) \; a = L^M b, \quad (\mathrm{ii}) \; \supp(L^k b) \subseteq \ball \cap O, \quad (\mathrm{iii}) \; \| (r^2 L)^k b \|_2 \leq r^{2M} |\ball \cap O|^{-\frac{1}{2}},
		\end{align}
		where $k = 0, \dots, M$.

		The atomic pre-Hardy space $\mathbb{H}^1_{L,\mathrm{at},M}$ consists of all $f \in \L^2$ that admit a $(1,2,M)$-atomic representation $f = \sum_j \lambda_j a_j$, where $(\lambda_j)_j \in \ell^1$, each $a_j$ is a $(1,2,M)$-atom, and the sum converges in $\L^2$. The space $\mathbb{H}^1_{L,\mathrm{at},M}$ is equipped with the quotient norm
		\begin{align}
			\| f \|_{\mathbb{H}^1_{L,\mathrm{at},M}} \coloneqq \inf \bigl\{ \| (\lambda_j)_j \|_{\ell^1} \colon f = \sum_j \lambda_j a_j \text{ is a } (1,2,M) \text{-atomic representation} \bigr\}.
		\end{align}
	\end{definition}

	\begin{remark}
		By (i), $\mathbb{H}^1_{L,\mathrm{at},M} \subseteq \overline{\sR(L)}$, in accordance with the definition of $\mathbb{H}^1_L$.
	\end{remark}

	\begin{remark}[comparison with~\cite{JFA-Hardy}]
		\label{rem:L_atomic_2}
		Definition~\ref{def:L_adapted_atoms} coincides with the definition given in~\cite[Def.~1.1 \&~1.2]{JFA-Hardy} up to two remarks.

		\begin{enumerate}
			\item The definition in~\cite{JFA-Hardy} is formulated with $(1,\infty,M)$-atoms, in which condition~(iii) is formulated with a size condition in $\L^\infty$. However, the results of~\cite{JFA-Hardy} are true for $(1,q,M)$-atoms if $1<q\leq \infty$. Indeed, having a $(1,\infty,M)$-atomic representation is stronger than having a $(1,q,M)$-atomic representation and the inclusion $\mathbb{H}^1_{L,\mathrm{at},M} \subseteq \mathbb{H}^1_{L,\mathrm{max}}$ was already formulated for more general $q$.

			\item If $O$ has finite measure, then $\mathbb{H}^1_{L,\mathrm{at},M}$ is replaced\footnote{In the published version, $\ker(L)$ is replaced by $\C$. This is true in their applications but has to be changed in the abstract formulation.} by $\mathbb{H}^1_{L,\mathrm{at},M} \oplus \ker(L)$ in~\cite{JFA-Hardy}. In our case, $\ker(L) = \ell^1(\ISet)$. We will have to take this into account by a case distinction later on.

		\end{enumerate}

	\end{remark}

	The following is one of the main results of~\cite{Hardy-Gaffney}, see \cite[Prop.~4.16]{Hardy-Gaffney}.

	\begin{proposition}[$L$-atomic characterization of $\mathbb{H}^1_L$]
		Assume $(\mathrm{G}(\mu))$ and that $L$ is self-adjoint.
		For $M > \nicefrac{d}{4}$ it holds $$\mathbb{H}^1_{L,\mathrm{at},M} = \mathbb{H}^1_L.$$
	\end{proposition}

	The inclusion $\mathbb{H}^1_{L,\mathrm{at},M} \subseteq \mathbb{H}^1_L$ is the simpler one and similar to the arguments in Section~\ref{Section: Inclusion of H_D^1(O) into Operator-adapted space}. For the converse inclusion, they first derive a so-called molecular decomposition. Molecules are a generalization of $(1,2,M)$-atoms in which the function $a$ is not supported in $\ball \cap O$ anymore, but has rapid decay over annuli around $\ball$. This does not yet require self-adjointness of $L$ but a bounded $\H^\infty$-calculus on $\overline{\sR(L)}$ is sufficient. This is always the case for $L$ by the Crozeix--Delyon theorem and does not require self-adjointness.
	Later, they decompose $L$-molecules into $L$-atoms. Only this latter part relies heavily on self-adjointness.

	We record a lemma that was already mentioned and used in Section~\ref{Section: Identification of operator-adapted and atomic Hardy spaces}.

	\begin{lemma}
		\label{lem:pre-adapted_embed_L1}
	     One has the estimate
		\begin{align}
			\| f \|_1 \lesssim \| S_L(f) \|_1 \mathrlap{\qquad (f\in \mathbb{H}_L^1).}
		\end{align}
		In other words, $\mathbb{H}^1_L \subseteq \L^1$ with continuous inclusion.
	\end{lemma}

	The argument is well-known to the expert and was also mentioned in~\cite[p.~75]{Hardy-Gaffney}.

	\begin{proof}
		Take $k=M$ in (iii) in the definition of a $(1,2,M)$-atom associated to $L$ to find $\| a \|_2 \leq |\ball \cap O|^{-\nicefrac{1}{2}}$. Thus, $\mathbb{H}^1_{L,\mathrm{at},M} \subseteq \L^1$ follows with the same argument as for traditional atomic Hardy spaces. The argument using a molecular decomposition works the same up to an additional summation over the annuli. Therefore, the result is true without assuming that $L$ is self-adjoint.
	\end{proof}

	To conclude the second claim from the outline we use~\cite[Cor.~1.5]{JFA-Hardy}. If $O$ has finite measure, this is their main result~\cite[Thm.~1.4]{JFA-Hardy}. The case of infinite measure was already known in the literature, see their references. Recall also the second part from Remark~\ref{rem:L_atomic_2}.

	\begin{proposition}[maximal characterization]
		Assume that $L$ is self-adjoint and satisfies $(\mathrm{G}(\mu))$ as well as the geometric assumptions $(\mathrm{UITC})$, $(D)$, $(\mathrm{Fat})$ and $(\mathrm{LU})$.
		Let $M \in \N$. We have the identities  $$\mathbb{H}^1_{L,\mathrm{at},M} \oplus \ell^1(\ISet) = \mathbb{H}^1_{L,\mathrm{max}} = \mathbb{H}^1_{L,\mathrm{rad}}.$$
	\end{proposition}

	Let us record the following corollary, which is the second claim from the outline.

	\begin{corollary}
		\label{cor:max_equal_atomic}
		One has the identities $$\mathbb{H}^1_L \oplus \ell^1(\ISet) = \mathbb{H}^1_{L,\mathrm{max}} = \mathbb{H}^1_{L,\mathrm{rad}}.$$
	\end{corollary}

	From Proposition~\ref{prop:completion} and Corollary~\ref{cor:max_equal_atomic} we conclude Theorem~\ref{thm:max}.

	\begin{proof}[\rm\bf{Proof of Theorem~\ref{thm:max}}]
		We perform the argument sketched in the outline. Corollary~\ref{cor:max_equal_atomic} yields for instance
		\begin{align}
			\label{eq:sqf_equal_max}
			\mathbb{H}^1_L \oplus \ell^1(\ISet) = \mathbb{H}^1_{L,\mathrm{max}}.
		\end{align}
		We take the completion of both sides of~\eqref{eq:sqf_equal_max}.
		According to Proposition~\ref{prop:completion}, $\H^1_L$ is a completion of $\mathbb{H}^1_L$. Since we can take the completion componentwise, $\H^1_L \oplus \ell^1(\ISet)$ is a completion of the left-hand side of~\eqref{eq:sqf_equal_max}. But by definition, $\H^1_{L,\mathrm{max}}$ is a completion of the right-hand side of~\eqref{eq:sqf_equal_max}, so $\H^1_L \oplus \ell^1(\ISet) = \H^1_{L,\mathrm{max}}$ follows. The argument for $\H^1_{L,\mathrm{rad}}$ is the same.
	\end{proof}

	\subsection{Application to the negative Laplacian}
	\label{subsec:appl_laplace}

	We apply Theorem~\ref{thm:max} to $-\Delta_D$ defined on a bounded and open set $O$ satisfying our geometric assumptions to conclude Theorem~\ref{thm:laplace}.

	\begin{proof}[\rm\bf{Proof of Theorem~\ref{thm:laplace}}]
		We appeal to Theorem~\ref{thm:max}. The geometric requirements are all fulfilled by assumption. Also, $-\Delta_D$ is self-adjoint. To invoke the theorem, it remains to check $\mathrm{G}(\mu)$. Since the coefficients of $-\Delta_D$ are real, this follows from~\cite[Thm.~1.2 \& Thm.~1.1]{BCE-Gauss-for-MBC}. Note that in this paper $\mathrm{G}(\mu)$ is defined with a growth factor $\e^{\omega t}$, where $\omega$ can be strictly positive. This is why we work with $O$ bounded. %
		Therefore, Theorem~\ref{thm:max} can be applied and yields all identities except
		\begin{align}
			\{ f\in \L^1 \colon S_{-\Delta_D}(f) \in \L^1 \} = \H^1_{-\Delta_D} \oplus \ell^1(\ISet).
		\end{align}
		Since $S_{-\Delta_D}(c) = 0$ for all $c \in \ell^1(\ISet)$ by the conservation property on $\ON$, it follows that the mapping $$f \mapsto \bigg(f - \sum_m \1_{O_m} (f)_{O_m} \bigg) \oplus ((f)_{O_m})_m$$ is an isomorphism between both spaces.
	\end{proof}


\begin{thebibliography}{10}
		\providecommand{\url}[1]{{\tt #1}}
		\providecommand{\urlprefix}{URL}
		\providecommand{\eprint}[2][]{\url{#2}}

		\bibitem{Hedberg}
		\textsc{D.~Adams} and \textsc{L.~Hedberg}.
		\newblock Function spaces and potential theory, Grundlehren der Mathematischen Wissenschaften, vol.~314,
		\newblock Springer, Berlin, 1995.

		\bibitem{Amenta-Tent-Spaces}
		\textsc{A.~Amenta}.
		{\em Tent spaces over metric measure spaces under doubling and related assumptions}.
		\newblock Operator theory in harmonic and non-commutative analysis. 23rd international workshop in operator theory and its applications, IWOTA, Sydney, Australia, July 16--20, 2012, Cham: Birkh{\"a}user/Springer (2014), 1--29.

		\bibitem{Amenta-Auscher}
		\textsc{A.~Amenta} and \textsc{P.~Auscher}.
		{\em Elliptic Boundary Value Problems with Fractional Regularity Data: The First Order Approach}. CRM Monogr. Ser., American Mathematical Society, \textbf{37} (2018).

		\bibitem{Auscher_Heat-Kernel}
		\textsc{P.~Auscher}.
		{\em Regularity theorems and heat kernel for elliptic operators}. J. Lond. Math. Soc., II. Ser. \textbf{54} (1996), no.~2, 284--296.

		\bibitem{AE-mixed}
		\textsc{P.~Auscher} and \textsc{M.~Egert}.
		\newblock {\em Mixed boundary value problems on cylindrical domains\/}.
		\newblock Adv. Differential Eq. \textbf{22} (2017), no.~1-2, 101--168.

		\bibitem{Block}
		\textsc{P.~Auscher} and \textsc{M.~Egert}.
		\newblock Boundary Value Problems and Hardy Spaces for Elliptic Systems with Block Structure. Progress in Mathematics, vol.~346,
		\newblock Birkhäuser, Cham, 2023.

	    \bibitem{Auscher-Russ}
		\textsc{P.~Auscher} and \textsc{E.~Russ}.
		{\em Hardy spaces and divergence operators on strongly {Lipschitz} domains of {{\(\mathbb R^n\)}}}. J. Funct. Anal. \textbf{201} (2003), no.~1, 148--184.


		\bibitem{Auscher-Stahlhut}
		\textsc{P.~Auscher} and \textsc{S.~Stahlhut}.
		{\em Functional calculus for first order systems of {Dirac} type and boundary value problems.} M{\'e}m. Soc. Math. Fr., Nouv. S{\'e}r. \textbf{144} (2016), 1--127, 157--164.



		\bibitem{Auscher_Tchamitchian_Domains}
		\textsc{P.~Auscher} and \textsc{P.~Tchamitchian}.
		\newblock Gaussian estimates for second order elliptic divergence operators on Lipschitz and $C^1$ domains, Evolution equations and their applications in physical and life sciences (Bad Herrenalb, 1998), Lecture Notes in Pure and Appl.~Math., vol.~215, 15--32,
		\newblock Dekker, New York, 2001.

		\bibitem{Bechtel_PhD}
		\textsc{S.~Bechtel}.
		\newblock On mixed boundary conditions, function spaces, and Kato's square root problem, PhD Thesis,
		\newblock TU Darmstadt, 2021.

		\bibitem{Kato_Mixed}
		\textsc{S.~Bechtel}, \textsc{M.~Egert} and \textsc{R.~Haller--Dintelmann}.
		{\em The Kato square root problem on locally uniform domains}. Adv. Math., \textbf{375} (2020), id/no. 107410, 37.


		\bibitem{BCE-Gauss-for-MBC}
		\textsc{T.~Böhnlein}, \textsc{S.~Ciani} and \textsc{M.~Egert}.
		{\em Gaussian estimates vs.~elliptic regularity on open sets}. Preprint (2023),
		\url{https://arxiv.org/abs/2307.03648v2}.

		\bibitem{JFA-Hardy}
		\textsc{T.~Bui}, \textsc{X.~Duong}, and \textsc{F.~Ly}.
		\newblock {\em Maximal function characterizations for Hardy spaces on spaces of homogeneous type with finite measure and applications.} Journal of Functional Analysis. \textbf{278} (2020), 1--55.


		\bibitem{Carbonaro-Dragicevic_Sg_Max_Operator}
		\textsc{A.~Carbonaro} and \textsc{O.~Dragičević}.
		{\em On semigroup maximal operators associated with divergence-form operators with complex coefficients}. Preprint (2022),
		\url{https://arxiv.org/abs/2207.11045v2}.

		\bibitem{CKS}
		\textsc{D-C.~Chang}, \textsc{S.~Krantz}, and \textsc{E.~Stein}.
		\newblock {\em $H^p$ Theory on a Smooth Domain in $\R^N$ and Elliptic Boundary Value Problems\/}.
		\newblock Journal of Functional Analysis \textbf{114} (1993), 286--347.

		\bibitem{Coifman-Weiss}
		\textsc{R.~Coifman} and \textsc{G.~Weiss}.
		{\em Extensions of Hardy spaces and their use in analysis}. Bull. Amer. Math. Soc. \textbf{83} (1977), no.~4, 569--645.


		\bibitem{Fattorini-Subordination}
		\textsc{H.~Fattorini}.
		\newblock Second Order Linear Differential Equations in Banach Spaces, North-Holland Math. Stud., vol.~108,
		\newblock Elsevier, Amsterdam, 1985.

		\bibitem{Fefferman-BMO-H1-duality}
		\textsc{C.~Fefferman}.
		{\em Characterizations of bounded mean oscillation}. Bull. Amer. Math. Soc. \textbf{77} (1971), 587--588.

		\bibitem{Hardy-Gaffney}
		\textsc{S.~Hofmann}, \textsc{G.~Lu}, \textsc{D.~Mitrea}, \textsc{M.~Mitrea}, and \textsc{L.~Yan}.
		{\em Hardy spaces associated to non-negative self-adjoint operators satisfying {D}avies-{G}affney estimates}.
		\newblock Mem. Amer. Math. Soc. \textbf{1214} (2011), vi+78 pp.

	    \bibitem{Jones}
		\textsc{P.~Jones}.
		{\em Quasiconformal mappings and extendability of functions in {Sobolev} spaces}. Acta Math. \textbf{147} (1981), 71--88.

		\bibitem{Kato}
		\textsc{T.~Kato}.
		\newblock Perturbation theory for linear operators, Classics in Mathematics,
		\newblock Springer-Verlag, Berlin, 1995.


		\bibitem{MMM}
		\textsc{D.~Mitrea}, \textsc{I.~Mitrea}, and \textsc{M.~Mitrea}.
		\newblock Geometric Harmonic Analysis I. A {S}harp {D}ivergence {T}heorem with {N}ontangential {P}ointwise {T}races. Developments in Mathematics, vol.~72,
		\newblock Springer, Cham, 2022.

	    \bibitem{Miyachi-Atoms}
		\textsc{A.~Miyachi}.
		{\em H{{\({}^ p\)}} spaces over open subsets of {{\(R^ n\)}}}. Stud. Math. \textbf{96} (1990), no.~3, 205--228.


        \bibitem{Werner_FA}
        \textsc{D.~Werner}.
        \newblock Funktional-analysis, ed.~8th,
        \newblock Springer-Verlag, Berlin, 2018.

		\bibitem{Zorin-Kranich_Master-Thesis}
		\textsc{P.~Zorin--Kranich}.
		{\em Interpolation between Banach spaces and continuity of Radon-like integral transforms}. Diploma thesis (2011),
		\url{https://arxiv.org/pdf/1301.1025.pdf}.

	\end{thebibliography}
\end{document}